\documentclass[11pt,leqno]{article}
\usepackage{amsfonts}
\usepackage{amsmath}
\usepackage{amssymb}
\usepackage{amsthm}
\usepackage{siunitx}
\usepackage[dvipdfmx]{graphicx}
\usepackage{mathrsfs,epsf}

\newtheorem{theorem}{\bf Theorem}[section]
\newtheorem{lemma}[theorem]{\bf Lemma}

\newtheorem{proposition}[theorem]{\bf Proposition}
\newtheorem{remark}[theorem]{\bf Remark}

\numberwithin{equation}{section}

\makeatletter
\newcommand{\opnorm}{\@ifstar\@opnorms\@opnorm}
\newcommand{\@opnorms}[1]{%
  \left|\mkern-1.5mu\left|\mkern-1.5mu\left|
   #1
  \right|\mkern-1.5mu\right|\mkern-1.5mu\right|
}
\newcommand{\@opnorm}[2][]{%
  \mathopen{#1|\mkern-1.5mu#1|\mkern-1.5mu#1|}
  #2
  \mathclose{#1|\mkern-1.5mu#1|\mkern-1.5mu#1|}
}

\begin{document}
\vspace*{0ex}
\begin{center}
{\Large\bf
Solitary wave solutions to the Isobe--Kakinuma model \\[0.5ex]
for water waves
}
\end{center}

\begin{center}
Mathieu Colin and Tatsuo Iguchi
\end{center}

\begin{abstract}
We consider the Isobe--Kakinuma model for two-dimensional water waves in the case of the flat bottom. 
The Isobe--Kakinuma model is a system of Euler--Lagrange equations 
for a Lagrangian approximating Luke's Lagrangian for water waves.
We show theoretically the existence of a family of small amplitude solitary wave solutions 
to the Isobe--Kakinuma model in the long wave regime. 
Numerical analysis for large amplitude solitary wave solutions is also provided and suggests 
the existence of a solitary wave of extreme form with a sharp crest. 
\end{abstract}

\section{Introduction}
\label{section:intro}
In this paper we consider the motion of two-dimensional water waves in the case of the flat bottom. 
The water wave problem is mathematically formulated as a free boundary problem for an irrotational flow 
of an inviscid and incompressible fluid under the gravitational field. 
Let $t$ be the time and $(x,z)$ the spatial coordinates. 
We assume that the water surface and the bottom are represented as $z = \eta(x,t)$ and $z = - h$, respectively. 
As was shown by J. C. Luke \cite{Luke1967}, the water wave problem has a variational structure. 
His Lagrangian density is of the form 
\begin{equation}\label{intro:Luke's Lagrangian}
\mathscr{L}_{\rm Luke}(\Phi,\eta) = \int_{-h}^{\eta(x,t)}\biggl(\partial_t\Phi(x,z,t)
 +\frac12 \bigl( (\partial_x\Phi(x,z,t))^2 + (\partial_z\Phi(x,z,t))^2 \bigr) + gz\biggr){\rm d}z,
\end{equation}
where $\Phi = \Phi(x,z,t)$ is the velocity potential and $g$ is the gravitational constant. 
M. Isobe \cite{Isobe1994, Isobe1994-2} and T. Kakinuma \cite{Kakinuma2000, Kakinuma2001, Kakinuma2003} 
proposed a model for water waves as a system of Euler--Lagrange equations for an approximate Lagrangian, 
which is derived from Luke's Lagrangian by approximating the velocity potential $\Phi$ 
in the Lagrangian appropriately. 
In this paper, we adopt an approximation under the form
\begin{equation}\label{intro:app}
\Phi(x,z,t) \simeq \sum_{i=0}^N(z+h)^{p_i}\phi_i(x,t),
\end{equation}
where $p_0,p_1,\ldots,p_N$ are nonnegative integers satisfying $0=p_0<p_1<\cdots<p_N$. 
Then, the corresponding Isobe--Kakinuma model in a nondimensional form is written as 
\begin{equation}\label{intro:IK}
\left\{
 \begin{array}{l}
  \displaystyle
  H^{p_i} \partial_t \eta + \sum_{j=0}^N \biggl\{ \partial_x \biggl(
   \frac{1}{p_i+p_j+1} H^{p_i+p_j+1} \partial_x\phi_j \biggr) 
   - \delta^{-2}\frac{p_ip_j}{p_i+p_j-1} H^{p_i+p_j-1} \phi_j \biggr\} = 0 \\
  \makebox[26em]{}\mbox{for}\quad i=0,1,\ldots,N, \\[1ex]
  \displaystyle
  \sum_{j=0}^N H^{p_j} \partial_t \phi_j + \eta 
   + \frac12 \left\{ \left( \sum_{j=0}^N H^{p_j}\partial_x\phi_j \right)^2 
   + \delta^{-2}\left( \sum_{j=0}^N p_j H^{p_j-1} \phi_j \right)^2 \right\} = 0,
 \end{array}
\right.
\end{equation}
where $H(x,t)=1+\eta(x,t)$ is the normalized depth of the water and $\delta$ a nondimensional parameter 
defined by the ratio of the mean depth $h$ to the typical wavelength $\lambda$. 
Here and in what follows we use the notational convention $\frac{0}{0} = 0$. 
For the derivation and basic properties of this model, we refer to 
Y. Murakami and T. Iguchi \cite{MurakamiIguchi2015} and R. Nemoto and T. Iguchi \cite{NemotoIguchi2017}. 
Moreover, it was shown by T. Iguchi \cite{Iguchi2018-1, Iguchi2018-2} that the Isobe--Kakinuma model 
\eqref{intro:IK} is a higher order shallow water approximation for the water wave problem in the strongly 
nonlinear regime. 
We note also that the Isobe--Kakinuma model \eqref{intro:IK} in the case $N=0$ is exactly the same as 
the shallow water equations. 
In the sequel, we always assume $N\geq1$.

In this paper, we look for solitary wave solutions to this model under the form 
\begin{equation}\label{formsol}
\eta=\eta(x+ct),\;\;\phi_j=\phi_j(x+ct),\;\;j=0,1,\ldots,N,
\end{equation}
where $c\in\mathbf{R}$ is an unknown constant phase speed. 
Plugging \eqref{formsol} into \eqref{intro:IK}, 
we obtain a system of ordinary differential equations 
\begin{equation}\label{intro:IK2}
\left\{
 \begin{array}{l}
  \displaystyle
  cH^{p_i} \eta' + \sum_{j=0}^N \biggl\{ \biggl(
   \frac{1}{p_i+p_j+1} H^{p_i+p_j+1} \phi_j' \biggr)' 
   - \delta^{-2}\frac{p_ip_j}{p_i+p_j-1} H^{p_i+p_j-1} \phi_j \biggr\} = 0 \\
  \makebox[24em]{}\mbox{for}\quad i=0,1,\ldots,N, \\[1ex]
  \displaystyle
  c\sum_{j=0}^N H^{p_j} \phi_j' + \eta 
   + \frac12 \left\{ \left( \sum_{j=0}^N H^{p_j} \phi_j' \right)^2 
   + \delta^{-2}\left( \sum_{j=0}^N p_j H^{p_j-1} \phi_j \right)^2 \right\} = 0.
 \end{array}
\right.
\end{equation}
As expected, this system has a variational structure, that is, the solution of this system is 
obtained as a critical point of the functional 
\[
\mathscr{L}_{\rm IK}(\phi_0,\ldots,\phi_N,\eta)
= c\mathscr{M}_{\rm IK}(\phi_0,\ldots,\phi_N,\eta) + \mathscr{E}_{\rm IK}(\phi_0,\ldots,\phi_N,\eta), 
\]
where 

\begin{align*}
\mathscr{M}_{\rm IK}(\phi_0,\ldots,\phi_N,\eta)
&= \int_{\mathbf{R}}\eta(x)\partial_x\bigl( \Phi^{\rm app}(x,\eta(x)) \bigr){\rm d}x, \\
\mathscr{E}_{\rm IK}(\phi_0,\ldots,\phi_N,\eta)
&= \frac12\int_{\mathbf{R}}\left\{
  \int_{-1}^{\eta(x)}\bigl( (\partial_x\Phi^{\rm app}(x,z))^2 
   + (\delta^{-1}\partial_z\Phi^{\rm app}(x,z))^2 \bigr) {\rm d}z + \eta(x)^2 \right\}{\rm d}x,
\end{align*}
and $\Phi^{\rm app}$ is the approximate velocity potential defined by 
\[
\Phi^{\rm app}(x,z) = \sum_{i=0}^N(z+1)^{p_i}\phi_i(x,t).
\]
We note that $\mathscr{M}_{\rm IK}$ and $\mathscr{E}_{\rm IK}$ represent the momentum in the 
horizontal direction and the total energy of the water, respectively. 
Both of them are conserved quantities for the Isobe--Kakinuma model \eqref{intro:IK}. 
In this paper we do not use this variational structure to construct solitary wave solutions 
to the Isobe--Kakinuma model, whereas we use a perturbation method with respect to the 
small nondimensional parameter $\delta$ in the long wave regime.

In order to give one of our main results in this paper concerning the existence of a family 
of solutions to \eqref{intro:IK2}, we introduce norms $\|\cdot\|$ and $\|\cdot\|_k$ for $k=0,1,2,\ldots,$ by 
\begin{equation}\label{intro:norm}
\|u\| = \sup_{x\in\mathbf{R}}\mbox{e}^{|x|}|u(x)|, \qquad
\|u\|_k = \sum_{j=0}^k \|u^{(j)}\|,
\end{equation}
where $u^{(j)}$ is the $j$-th order derivative of $u$. 
We also introduce the function spaces $B_e^k$ and $B_o^k$ as closed subspaces of all even and odd 
functions $u\in C^k(\mathbf{R})$ satisfying $\|u\|_k <+\infty $, respectively, 
equipped with the norm $\|\cdot\|_k$, and put 
$B_\alpha^\infty=\cap_{k=0}^\infty B_\alpha^k$ for $\alpha=e,o$. 
The following theorem guarantees the existence of small amplitude solitary wave solutions to the 
Isobe--Kakinuma model in the long wave regime.

\begin{theorem}\label{intro:theorem}
There exists a positive constant $\delta_0$ such that for any $\delta\in(0,\delta_0]$ 
the Isobe--Kakinuma model \eqref{intro:IK2} has a solution 
$(c^\delta,\eta^\delta,\phi_0^\delta,\ldots,\phi_N^\delta)$, which satisfies 
$\eta^\delta,\phi_0^{\delta \prime} \in B_e^\infty$ and 
$\phi_1^\delta,\ldots,\phi_N^\delta \in B_o^\infty$. 
Moreover, the solution satisfies $c^\delta=1+2\gamma\delta^2$ and 
\[
\begin{cases}\label{est:th1}
 \|\eta^\delta - 4\gamma\delta^2\mbox{\rm sech}^2x\|_k
  + \|(\phi_0^\delta + 4\gamma\delta^2\tanh x)'\|_k \leq C_k\delta^4, \\
 \|\phi_j^\delta - 4\gamma\gamma_j\delta^4\mbox{\rm sech}^2x\tanh x\|_k \leq C_k\delta^6
  \quad\mbox{for}\quad j=1,\ldots,N
\end{cases}
\]
for any $k\in \mathbf{N}$, where the constants $\gamma,\gamma_1,\ldots,\gamma_N$ are determined through 
$p_1,\ldots,p_N$ and the constant $C_k$ does not depend on $\delta$ but on $k$. 
\end{theorem}

\begin{remark}\label{intro:remark0}
The constants $\gamma$ and $\boldsymbol{\gamma}=(\gamma_1,\ldots,\gamma_N)^{\rm T}$ in the statement 
of Theorem \ref{intro:theorem} are given by 
$\boldsymbol{\gamma}=A_1^{-1}(\boldsymbol{1}-\boldsymbol{a}_0)$ and 
$\gamma=(\boldsymbol{1}-\boldsymbol{a}_0)\cdot A_1^{-1}(\boldsymbol{1}-\boldsymbol{a}_0)$ with 
a $N\times N$ matrix $A_1$ and a vector $\boldsymbol{a}_0$ given by 
\begin{equation}\label{as:defA}
A_1 = \biggl(\frac{p_ip_j}{p_i+p_j-1}\biggr)_{1\leq i,j\leq N}, \quad
\begin{pmatrix}
 1 & \boldsymbol{a}_0^{\rm T} \\
 \boldsymbol{a}_0 & A_0
\end{pmatrix}
= \biggl(\frac{1}{p_i+p_j+1}\biggr)_{0\leq i,j\leq N},
\end{equation}
and $\boldsymbol{1}=(1,\ldots,1)^{\rm T}$. 
The matrices $A_0$ and $A_1$ are positive, so that the constant $\gamma$ is also positive. 
We will use these notations throughout this paper. 
\end{remark}

\begin{remark}\label{intro:remark1}
This theorem ensures the existence of a family of solitary wave solutions to the Isobe--Kakinuma model 
traveling to the left. 
We can also show a similar existence theorem to solitary wave solutions to the model traveling to the right. 
\end{remark}

\begin{remark}\label{intro:remark2}
By neglecting the term of order $O(\delta^4)$, the solitary wave solutions in the dimensional form 
are given by 
\[
\begin{cases}
 c \simeq \bigl( 1+\frac{a}{2h} \bigr)\sqrt{gh}, \\
 \eta(x,t) \simeq a\,\mbox{\rm sech}^2\Bigl(\sqrt{\frac{a}{4\gamma h^3}}(x+ct)\Bigr), \\
 \phi_0(x,t) \simeq -\sqrt{4\gamma gah^2}\tanh\Bigl(\sqrt{\frac{a}{4\gamma h^3}}(x+ct)\Bigr), \\
 \phi_j(x,t) \simeq 0  \quad\mbox{for}\quad j=1,\ldots,N,
\end{cases}
\]
where $g$ is the gravitational constant and $a$ is the amplitude of the wave. 
\end{remark}

In this paper we also analyze numerically the existence of large amplitude solitary wave solutions to 
the Isobe--Kakinuma model in a special case where we choose the parameters as $N=1$ and $p_1=2$, 
that is, \eqref{na:eq1}. 
We note that even in this simplest case the Isobe--Kakinuma model gives a better approximation 
than the well-known Green--Naghdi equations in the shallow water and strongly nonlinear regime. 
See T. Iguchi \cite{Iguchi2018-1, Iguchi2018-2}. 
Numerical analysis suggests that there exists a critical value of $\delta$ given approximately by 
\begin{equation}\label{intro:cv}
\delta_c = 0.62633493
\end{equation}
such that for any $\delta\in(0,\delta_c)$, the Isobe--Kakinuma model \eqref{na:eq1} admits a smooth 
solitary wave solution and that this family of waves converges to a solitary one of extreme form 
with a shape crest as $\delta\uparrow\delta_c$. 
Moreover, the included angle $\theta$ of the crest in the physical space is given approximately by 
\begin{equation}\label{intro:angle}
\theta = \ang{152.6}.
\end{equation}
See Figure \ref{intro:extreme wave}.

\begin{figure}[ht]
\setlength{\unitlength}{1pt}
\begin{picture}(0,0)
\put(225,-37){$\theta$}
\end{picture}
\begin{center}
\includegraphics[width=0.7\linewidth]{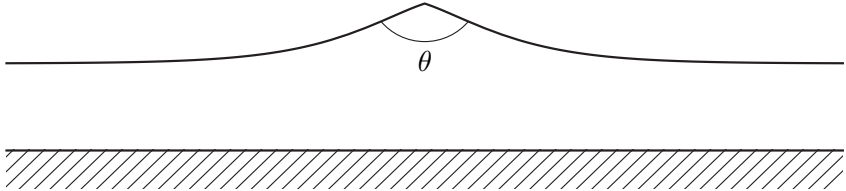}
\end{center}
\caption{Solitary wave of extreme form with the included angle $\theta=\ang{152.6}$.}
\label{intro:extreme wave}
\end{figure}

Here, we mention the related results on the existence of the solitary wave solutions to the water 
wave problem. 
The existence of small amplitude solitary waves for the water wave problem was first given by 
K. O. Friedrichs and D. H. Hyers \cite{FriedrichsHyers1954}. 
Then, C. J. Amick and J. F. Toland \cite {AmickToland1981} proved the existence of solitary wave 
solutions of all amplitudes from zero up to and including that of the solitary wave of greatest height. 
On the other hand, a periodic wave of permanent form is called Stokes' wave. 
The existence of Stokes's wave of extreme form as well as the sharp crest of the included angle $\ang{120}$ 
was predicted by G. G. Stokes \cite{Stokes1847}, and then proved theoretically by 
C. J. Amick, L. E. Fraenkel, and J. F. Toland \cite{AmickFraenkelToland1982}. 
An extension of these existence theories to the water waves with vorticity was given by 
E. Varvaruca \cite{Varvaruca2009}. 
The author has proved the existence of the solitary wave as well as Stokes' wave of greatest height 
with a shape crest of the same included angle $\ang{120}$ as in the irrotational case. 
As for the model equations for water waves, it is well-known that the Korteweg--de Vries equation 
has solitary wave solutions of arbitrarily large amplitude and does not have any wave of extreme form. 
The Green--Naghdi equations are known as higher order shallow water approximate equations for water waves 
in the strongly nonlinear regime and have the same solitary wave solutions as those of 
the Korteweg--de Vries equation, but again does not have any solitary wave of extreme form. 
Compared to these models, the Isobe--Kakinuma model even in the simplest case catches the property 
on the existence of the solitary wave of extreme form, 
although the included angle of the crest is not $\ang{120}$. 
We also mention a result by D. Lannes and F. Marche \cite{LannesMarche2016}, 
where it was shown that an extended Green--Naghdi equations for the water waves with a constant vorticity 
has a solitary wave solution of extreme form.

The contents of this paper are as follows. 
In Section \ref{section:cl}, we give conservation laws for the Isobe--Kakinuma model, 
which will be used in the numerical analysis in Section \ref{section:na}. 
In Section \ref{section:as}, by using formal asymptotic analysis 
we calculate the first order approximate term with respect to $\delta^2$ in the expansion \eqref{as:expansion} of 
the solitary wave solution to the Isobe--Kakinuma model \eqref{intro:IK2}. 
In Section \ref{section:red}, we reduce the problem by deriving equations for the remainder terms. 
In Section \ref{section:gf}, we construct Green's functions for the linearized equations together with 
their estimates. 
In Section \ref{section:ET}, we first prove an existence theorem for the system of linear equations derived 
in Section \ref{section:red} and then prove the existence of small amplitude solitary wave solutions to the 
Isobe--Kakinuma model. 
Finally, in Section \ref{section:na}, we analyze numerically large amplitude solitary wave solutions 
and calculate the solitary wave of extreme form together with the included angle.

\newpage
\noindent
{\bf Acknowledgement} \\
T. I. was partially supported by JSPS KAKENHI Grant Number JP17K18742 and JP17H02856.

\section{Conservation laws}
\label{section:cl}
As was explained in the previous section, for the Isobe--Kakinuma model \eqref{intro:IK}, the mass, 
the momentum in the horizontal direction, and the total energy are conserved. 
In the numerical analysis which will be carried out in Section \ref{section:na}, 
we also need to know corresponding flux functions. 
The following proposition gives such flux functions.

\begin{proposition}\label{cl:prop1}
Any regular solution $(\eta,\phi_0,\ldots,\phi_N)$ to the Isobe--Kakinuma model \eqref{intro:IK} 
satisfies the conservation laws 
\begin{equation}\label{cl:mass}
\partial_t \eta
 + \partial_x \left( \sum_{j=0}^N \frac{1}{p_j+1} H^{p_{j+1}} \partial_x\phi_j\right) = 0, 
\end{equation}
\begin{align}\label{cl:momentum}
& \partial_t \left\{ \eta \partial_x \left( \sum_{i=0}^N H^{p_i} \phi_i \right) \right\} 
 + \partial_x \left\{
  -\eta \partial_t \left( \sum_{i=0}^N H^{p_i} \phi_i \right) - \frac12\eta^2 \right. \\
&\quad\left.
 + \frac12\sum_{i,j=0}^N\left( \frac{1}{p_i+p_j+1}H^{p_i+p_j+1}(\partial_x\phi_i)(\partial_x\phi_j) 
  - \delta^{-2}\frac{p_ip_j}{p_i+p_j-1}H^{p_i+p_j-1}\phi_i\phi_j  \right) \right\} = 0, \nonumber
\end{align}
and 
\begin{align}\label{cl:energy}
& \partial_t \left\{ \frac12\eta^2 
 + \frac12\sum_{i,j=0}^N\left( \frac{1}{p_i+p_j+1}H^{p_i+p_j+1}(\partial_x\phi_i)(\partial_x\phi_j) 
  + \delta^{-2}\frac{p_ip_j}{p_i+p_j-1}H^{p_i+p_j-1}\phi_i\phi_j  \right) \right\} \\
&\quad
 - \partial_x \left\{ \sum_{i,j}^N \frac{1}{p_i+p_j+1}H^{p_i+p_j+1}(\partial_x\phi_i)(\partial_t\phi_j)
  \right\} = 0. \nonumber
\end{align}
\end{proposition}

\begin{proof}[{\bf Proof}.]
Conservation law of mass \eqref{cl:mass} is nothing but the first equation in \eqref{intro:IK} with $i=0$. 
It follows from the equations in \eqref{intro:IK} that 
\begin{align*}
& \partial_t \left\{ \eta \partial_x \left( \sum_{i=0}^N H^{p_i} \phi_i \right) \right\} 
 - \partial_x \left\{
  \eta \partial_t \left( \sum_{i=0}^N H^{p_i} \phi_i \right) \right\} \\
&= (\partial_t \eta) \sum_{i=0}^N H^{p_i} \partial_x\phi_i
 - (\partial_x \eta) \sum_{i=0}^N H^{p_i} \partial_t\phi_i \\
&= \sum_{i=0}^N \left[ - \sum_{j=0}^N \biggl\{ \partial_x \biggl(
   \frac{1}{p_i+p_j+1} H^{p_i+p_j+1} \partial_x\phi_j \biggr) 
   - \delta^{-2}\frac{p_ip_j}{p_i+p_j-1} H^{p_i+p_j-1} \phi_j \biggr\} \right] \partial_x\phi_i \\
&\quad\;
 + (\partial_x \eta) \left[ \eta 
   + \frac12 \left\{ \left( \sum_{j=0}^N H^{p_j}\partial_x\phi_j \right)^2 
   + \delta^{-2}\left( \sum_{j=0}^N p_j H^{p_j-1} \phi_j \right)^2 \right\} \right] \\
&= -\sum_{i,j=0}^N \left( 
  \frac{1}{p_i+p_j+1} H^{p_i+p_j+1} (\partial_x\phi_i)(\partial_x^2\phi_j)
  - \delta^{-2}\frac{p_ip_j}{p_i+p_j-1} H^{p_i+p_j-1} \phi_j\partial_x\phi_i \right) \\
&\quad\;
 + (\partial_x \eta) \left\{ \eta - \frac12\sum_{i,j=0}^N \bigl(
  H^{p_i+p_j} (\partial_x\phi_i)(\partial_x\phi_j)
  - \delta^{-2}p_ip_jH^{p_i+p_j-2} \phi_i\phi_j \bigr) \right\} \\
&= \partial_x \left\{ \frac12\eta^2
 - \frac12\sum_{i,j=0}^N\left( \frac{1}{p_i+p_j+1}H^{p_i+p_j+1}(\partial_x\phi_i)(\partial_x\phi_j) 
  - \delta^{-2}\frac{p_ip_j}{p_i+p_j-1}H^{p_i+p_j-1}\phi_i\phi_j  \right) \right\},
\end{align*}
which gives conservation law of momentum \eqref{cl:momentum}. 
We see also that 
\begin{align*}
& \partial_t \left\{ \frac12\eta^2 
 + \frac12\sum_{i,j=0}^N\left( \frac{1}{p_i+p_j+1}H^{p_i+p_j+1}(\partial_x\phi_i)(\partial_x\phi_j) 
  + \delta^{-2}\frac{p_ip_j}{p_i+p_j-1}H^{p_i+p_j-1}\phi_i\phi_j  \right) \right\} \\
&= (\partial_t\eta)\left[ \eta 
   + \frac12 \left\{ \left( \sum_{j=0}^N H^{p_j}\partial_x\phi_j \right)^2 
   + \delta^{-2}\left( \sum_{j=0}^N p_j H^{p_j-1} \phi_j \right)^2 \right\} \right] \\
&\quad\;
 + \sum_{i,j=0}^N\left( \frac{1}{p_i+p_j+1}H^{p_i+p_j+1}(\partial_x\phi_i)(\partial_t\partial_x\phi_j) 
  + \delta^{-2}\frac{p_ip_j}{p_i+p_j-1}H^{p_i+p_j-1}\phi_i\partial_t\phi_j  \right) \\
&= -(\partial_t\eta) \sum_{j=0}^N H^{p_j}\partial_t\phi_j
 + \partial_x\left\{ 
  \sum_{i,j=0}^N\frac{1}{p_i+p_j+1}H^{p_i+p_j+1}(\partial_x\phi_i)(\partial_t\phi_j) \right\} \\
&\quad\;
 + \sum_{i,j=0}^N \left\{ \partial_x \biggl(
   \frac{1}{p_i+p_j+1} H^{p_i+p_j+1} \partial_x\phi_i \biggr) 
   - \delta^{-2}\frac{p_ip_j}{p_i+p_j-1} H^{p_i+p_j-1} \phi_i \right\} \partial_t\phi_j \\
&= \partial_x\left\{ 
  \sum_{i,j=0}^N\frac{1}{p_i+p_j+1}H^{p_i+p_j+1}(\partial_x\phi_i)(\partial_t\phi_j) \right\},
\end{align*}
which gives conservation law of energy \eqref{cl:energy}. 
\end{proof}

These conservation laws provide directly the following proposition.

\begin{proposition}\label{cl:prop2}
For any regular solution $(\eta,\phi_0,\ldots,\phi_N)$ to the Isobe--Kakinuma model \eqref{intro:IK2} 
satisfying the condition at spatial infinity 
\begin{equation}\label{cl:BCI}
\eta(x), \phi_0'(x),\ldots,\phi_N'(x),\phi_1(x),\ldots,\phi_N(x) \to 0 \quad\mbox{as}\quad x\to\pm\infty,
\end{equation}
we have 
\begin{equation}\label{cl:cl1}
c\eta + \sum_{j=0}^N \frac{1}{p_j+1} H^{p_j+1} \phi_j' = 0
\end{equation}
and 
\begin{equation}\label{cl:cl2}
\frac12\eta^2 
 - \frac12\sum_{i,j=0}^N\left( \frac{1}{p_i+p_j+1}H^{p_i+p_j+1}\phi_i'\phi_j' 
  - \delta^{-2}\frac{p_ip_j}{p_i+p_j-1}H^{p_i+p_j-1}\phi_i\phi_j  \right) = 0.
\end{equation}
\end{proposition}

\section{Approximate solutions}
\label{section:as}
We first transform the Isobe--Kakinuma model \eqref{intro:IK2} into an equivalent system. 
The first equation in \eqref{intro:IK2} with $i=0$ can be integrated under the condition at 
spatial infinity \eqref{cl:BCI} so that we have \eqref{cl:cl1}, particularly, 
\[
c\eta = -\sum_{j=0}^N \frac{1}{p_j+1} H^{p_j+1} \phi_j'.
\]
Plugging this into the first equation in \eqref{intro:IK2} with $i=1,\ldots,N$ to eliminate $c\eta$, we have 
\begin{align*}
&H^{p_i}\sum_{j=0}^N \biggl(
   \frac{1}{p_j+1} H^{p_j+1} \phi_j' \biggr)'  \\
&=
\sum_{j=0}^N \biggl\{ \biggl(
   \frac{1}{p_i+p_j+1} H^{p_i+p_j+1} \phi_j' \biggr)' 
   - \delta^{-2}\frac{p_ip_j}{p_i+p_j-1} H^{p_i+p_j-1} \phi_j \biggr\},
\end{align*}
which are equivalent to 
\begin{align*}
\sum_{j=0}^N \biggl( \frac{1}{p_j+1}-\frac{1}{p_i+p_j+1} \biggr)H^{p_j+2} \phi_j''
 + \delta^{-2}\sum_{j=0}^N \frac{p_ip_j}{p_i+p_j-1} H^{p_j} \phi_j = 0
\end{align*}
for $i=1,\ldots,N$. 
Therefore, \eqref{intro:IK2} under the condition \eqref{cl:BCI} is transformed equivalently into 
\begin{equation}\label{as:IK}
\left\{
 \begin{array}{l}
  \displaystyle
  c\eta + \sum_{j=0}^N \frac{1}{p_j+1} H^{p_j+1} \phi_j' = 0, \\
  \displaystyle
  \sum_{j=0}^N \biggl( \frac{1}{p_j+1}-\frac{1}{p_i+p_j+1} \biggr)H^{p_j+2}\phi_j''
   + \delta^{-2}\sum_{j=0}^N \frac{p_ip_j}{p_i+p_j-1} H^{p_j} \phi_j = 0 \\
  \makebox[22em]{}\mbox{for}\quad i = 1,\ldots,N, \\
  \displaystyle
  c\sum_{j=0}^N H^{p_j} \phi_j' + \eta 
   + \frac12 \left\{ \left( \sum_{j=0}^N H^{p_j}\phi_j' \right)^2 
   + \delta^{-2}\left( \sum_{j=0}^N p_j H^{p_j-1} \phi_j \right)^2 \right\} = 0,
 \end{array}
\right.
\end{equation}
where $H = 1+\eta$.

To simplify the notation, we put $\boldsymbol{\phi}=(\phi_1,\ldots,\phi_N)^{\rm T}$. 
Suppose that $(c,\eta,\phi_0,\boldsymbol{\phi})$ is a solution to \eqref{as:IK} and can be expanded 
with respect to $\delta^2$ as 
\begin{equation}\label{as:expansion}
\begin{cases}
c = c_{(0)} + \delta^2 c_{(1)} + \delta^4 c_{(2)} + \cdots, \\
\eta = \delta^2( \eta_{(0)} + \delta^2 \eta_{(1)} + \delta^4 \eta_{(2)} + \cdots), \\
\phi_0 = \delta^2( \phi_{0(0)} + \delta^2 \phi_{0(1)} + \delta^4 \phi_{0(2)} + \cdots), \\
\boldsymbol{\phi} = \delta^2( \boldsymbol{\phi}_{(0)} + \delta^2 \boldsymbol{\phi}_{(1)}
 + \delta^4 \boldsymbol{\phi}_{(2)} + \cdots). 
\end{cases}
\end{equation}

\begin{remark}
In this expansion, we assumed a priori that $\eta$, $\phi_0$, and $\boldsymbol{\phi}$ are of order $O(\delta^2)$. 
This is essentially the assumption that the solution is in the long wave regime. 
\end{remark}

\noindent
Then, equating the coefficients of the lowest power of $\delta^2$ we have 
$A_1\boldsymbol{\phi}_{(0)}=\boldsymbol{0}$ so that $\boldsymbol{\phi}_{(0)}=\boldsymbol{0}$, 
because $A_1$ is nonsingular. 
Equating the coefficients of $\delta^2$ we have 
\[
\begin{cases}
 c_{(0)}\eta_{(0)} + \phi_{0(0)}' = 0, \\
 (\boldsymbol{1}-\boldsymbol{a}_0)\phi_{0(0)}'' + A_1\boldsymbol{\phi}_{(1)} = \boldsymbol{0}, \\
 c_{(0)}\phi_{0(0)}' + \eta_{(0)} = 0.
\end{cases}
\]
It follows from the first and the third equations of the above system that $c_{(0)}^2=1$ in order to obtain a 
nontrivial solution. 
In the following we will consider the case $c_{(0)}=1$. 
Then, we have 
\begin{equation}\label{as:sol1}
\phi_{0(0)}' = -\eta_{(0)}, \quad
\boldsymbol{\phi}_{(1)} = -\boldsymbol{\gamma}\phi_{0(0)}'' =  \boldsymbol{\gamma}\eta_{(0)}',
\end{equation}
where $\boldsymbol{\gamma}$ is the constant vector defined in Remark \ref{intro:remark0}. 
Next, equating the coefficients of $\delta^4$ in the first and the third equations in \eqref{as:IK} we have 
\[
\begin{cases}
 \eta_{(1)} + \phi_{0(1)}' + c_{(1)}\eta_{(0)} + \eta_{(0)}\phi_{0(0)}'
  + \boldsymbol{a}_0\cdot\boldsymbol{\phi}_{(1)}' = 0, \\[0.5ex]
 \phi_{0(1)}' + \eta_{(1)} + c_{(1)}\phi_{0(0)}' + \boldsymbol{1}\cdot\boldsymbol{\phi}_{(1)}'
  + \frac12(\phi_{0(0)}')^2 = 0,
\end{cases}
\]
which together with \eqref{as:sol1} yield 
\begin{equation}\label{as:KdV}
2c_{(1)}\eta_{(0)} -\frac32\eta_{(0)}^2 - \gamma\eta_{(0)}'' = 0,
\end{equation}
where $\gamma$ is the positive constant defined in Remark \ref{intro:remark0}.

As is well known, \eqref{as:KdV} has a family of solutions of the form 
\[
\textstyle
c_{(1)}=\alpha, \qquad \eta_{(0)}(x) = 2\alpha\,\mbox{\rm sech}^2\bigl( \sqrt{\frac{\alpha}{2\gamma}}x \bigr)
\]
with a parameter $\alpha>0$. 
Then, it follows from \eqref{as:sol1} that 
\[
\begin{cases}
 \phi_{0(0)}(x) = -2\sqrt{2\alpha\gamma}\tanh\bigl( \sqrt{\frac{\alpha}{2\gamma}}x \bigr) + \mbox{\rm const.}, \\
 \boldsymbol{\phi}_{(1)}(x)
  = -4\alpha\boldsymbol{\gamma}\sqrt{\frac{\alpha}{2\gamma}}\tanh\bigl( \sqrt{\frac{\alpha}{2\gamma}}x \bigr)
  \,\mbox{\rm sech}^2\bigl( \sqrt{\frac{\alpha}{2\gamma}}x \bigr).
\end{cases}
\]
Therefore, it is natural to expect that \eqref{intro:IK2} has a family of solutions of the form 
\begin{equation}\label{as:exp}
\begin{cases}
c = 1 + \alpha \delta^2 + O(\delta^4), \\
\eta(x) = 2\alpha\delta^2\,\mbox{\rm sech}^2\bigl( \sqrt{\frac{\alpha}{2\gamma}}x \bigr) + O(\delta^4), \\
\phi_{0}(x) = -2\sqrt{2\alpha\gamma}\delta^2\tanh\bigl( \sqrt{\frac{\alpha}{2\gamma}}x \bigr) + O(\delta^4), \\
\boldsymbol{\phi}(x) = -4\alpha\boldsymbol{\gamma}\sqrt{\frac{\alpha}{2\gamma}}\delta^4
 \tanh\bigl( \sqrt{\frac{\alpha}{2\gamma}}x \bigr)
 \,\mbox{\rm sech}^2\bigl( \sqrt{\frac{\alpha}{2\gamma}}x \bigr) + O(\delta^6)
\end{cases}
\end{equation}
with parameters $\alpha>0$ and $0<\delta\ll1$.

\section{Reduction of the problem}
\label{section:red}
Without loss of generality, we can assume that $\alpha=2\gamma$. 
We denote the approximate solitary wave solution obtained in the previous section by $\eta_{(0)}$ and $\phi_{(0)}$, 
that is, 
\begin{equation}\label{ert:as}
\eta_{(0)}(x) = 4\gamma\,\mbox{\rm sech}^2 x, \qquad \phi_{(0)}(x) = -4\gamma\tanh x.
\end{equation}
Then, we will look for solutions to the Isobe--Kakinuma model \eqref{intro:IK2} in the form 
\begin{equation}\label{red:ansatz}
c = 1 + 2\gamma\delta^2, \quad \eta = \delta^2\eta_{(0)} + \delta^4\zeta, \quad 
\phi_0 = \delta^2\phi_{(0)} + \delta^4\psi_0, \quad 
\boldsymbol{\phi} = \delta^4\boldsymbol{\gamma}\eta_{(0)}' + \delta^6\boldsymbol{\psi}.
\end{equation}
Plugging these into the first and the third equations in \eqref{as:IK} we obtain 
\begin{equation}\label{red:IK1}
\begin{cases}
 \zeta + \psi_0' + G_1
  + \delta^2\bigl( (2\gamma-\eta_{(0)})\zeta + \eta_{(0)}\psi_0'
   + \boldsymbol{a}_0\cdot\boldsymbol{\psi}' + G_2 \bigr) + \delta^4 F_1 = 0, \\
 \zeta + \psi_0' + G_3
  + \delta^2\bigl( (2\gamma-\eta_{(0)})\psi_0' + \boldsymbol{1}\cdot\boldsymbol{\psi}' + G_4 \bigr)
   + \delta^4 F_2 = 0, \\
\end{cases}
\end{equation}
where 
\[
\begin{cases}
G_1 = 2\gamma\eta_{(0)} - \eta_{(0)}^2 + \kappa_1\eta_{(0)}'', \\
G_2 = \kappa_2\eta_{(0)}\eta_{(0)}'', \\
G_3 = -2\gamma\eta_{(0)} + \frac12\eta_{(0)}^2 + \kappa_2\eta_{(0)}'', \\
G_4 = \kappa_2(2\gamma - \eta_{(0)})\eta_{(0)}'' + \kappa_3\eta_{(0)}\eta_{(0)}'' + \frac12\kappa_3^2(\eta_{(0)}')^2,
\end{cases}
\]
with $\kappa_1=\boldsymbol{a}_0\cdot\boldsymbol{\gamma}$, $\kappa_2=\boldsymbol{1}\cdot\boldsymbol{\gamma}$, and 
$\kappa_3=(p_1,\ldots,p_N)^{\rm T}\cdot\boldsymbol{\gamma}$ and 
\begin{align*}
& F_1 = \zeta\psi_0' + \eta_{(0)}\sum_{j=1}^N\psi_j'
 + \sum_{j=1}^N\frac{1}{p_j+1}\delta^{-4}
  \bigl( H^{p_j+1} - 1 - \delta^2(p_j+1)\eta_{(0)}\bigr)(\gamma_j\eta_{(0)}'' + \delta^2\psi_j'), \\
& F_2 = F_3 + 2\gamma F_4 + \frac12(\psi_0'+\kappa_2\eta_{(0)}'')^2 \\
&\phantom{ F_2 = }
 + \frac12\delta^{-4}\bigl\{ \bigl( \phi_{(0)}' + \delta^2(\psi_0'+\kappa_2\eta_{(0)}'')
   + \delta^4 F_4 \bigr)^2
  - \bigl( \phi_{(0)}' + \delta^2(\psi_0'+\kappa_2\eta_{(0)}'') \bigr)^2 \bigr\} \\
&\phantom{ F_2 = }
 + \frac12\delta^{-2}\left\{
  \left(\sum_{j=1}^Np_jH^{p_j-1}(\gamma_j\eta_{(0)}'+\delta^2\psi_j)\right)^2
  - \left(\sum_{j=1}^Np_j\gamma_j\eta_{(0)}'\right)^2 \right\},
\end{align*}
with $H=1+\delta^2\eta_{(0)}+\delta^4\zeta$ and 
\begin{align*}
& F_3 = \eta_{(0)}\sum_{j=1}^Np_j\psi_j'
 + \sum_{j=1}^N\delta^{-4}(H^{p_j}-1-\delta^2 p_j\eta_{(0)})(\gamma_j\eta_{(0)}''+\delta^2\psi_j'), \\
& F_4 = \sum_{j=1}^N\psi_j'
 + \sum_{j=1}^N\delta^{-2}(H^{p_j}-1)(\gamma_j\eta_{(0)}''+\delta^2\psi_j').
\end{align*}
In the above calculations, we have used the identities 
\begin{align*}
\sum_{j=0}^NH^{p_j}\phi_j'
&= \delta^2\phi_{(0)}' + \delta^4( \psi_0'+\kappa_2\eta_{(0)}'')
  + \delta^6\left( \kappa_3\eta_{(0)}\eta_{(0)}''+\sum_{j=1}^N\psi_j' \right)
  + \delta^8 F_3 \\
&= \delta^2\phi_{(0)}' + \delta^4( \psi_0'+\kappa_2\eta_{(0)}'')
  + \delta^6 F_4.
\end{align*}
Since $\eta_{(0)}$ satisfies \eqref{as:KdV} with $c_{(1)}=2\gamma$, we have $G_1=G_3$, 
so that \eqref{red:IK1} is equivalent to 
\begin{equation}\label{red:IK2}
\begin{cases}
 (\eta_{(0)}-2\gamma)\zeta + 2(\gamma-\eta_{(0)})\psi_0'
  + (\boldsymbol{1}-\boldsymbol{a}_0)\cdot\boldsymbol{\psi}' = g_0 + \delta^2 f_0, \\
 \zeta + \psi_0' = g_{N+1} + \delta^2 f_{N+1},
\end{cases}
\end{equation}
with $g_0=G_2-G_4$, $g_{N+1}=G_1$, $f_0=F_1-F_2$, and 
$f_{N+1}=(\eta_{(0)}-2\gamma)\psi_0'-\boldsymbol{1}\cdot\boldsymbol{\psi}'-G_4-\delta^2 F_2$. 
On the other hand, plugging \eqref{red:ansatz} into the second equation in \eqref{as:IK}, we obtain 
\begin{equation}\label{red:IK3}
(\boldsymbol{1}-\boldsymbol{a}_0)\psi_0'' - \delta^2(A_0-\boldsymbol{1}\otimes\boldsymbol{a}_0)\boldsymbol{\psi}''
 + A_1\boldsymbol{\psi} = \boldsymbol{g}+\delta^2\boldsymbol{f},
\end{equation}
where 
\begin{align}
\boldsymbol{g} &= \bigl( 2(\boldsymbol{1}-\boldsymbol{a}_0)
 - A_1\,\mbox{\rm diag}(p_1,\ldots,p_N)\boldsymbol{\gamma} \bigr)\eta_{(0)}\eta_{(0)}'
 + (A_0-\boldsymbol{1}\otimes\boldsymbol{a}_0)\boldsymbol{\gamma}\eta_{(0)}''', \label{boldf}\\
\boldsymbol{f} &= (\boldsymbol{1}-\boldsymbol{a}_0)\bigl(
 \delta^{-4}(H^2-1-2\delta^2\eta_{(0)})\eta_{(0)}' + \delta^{-2}(H^2-1)\psi_0'' \bigr) \\
&\quad
 -(A_0-\boldsymbol{1}\otimes\boldsymbol{a}_0)\delta^{-2}\,
 \bigl( \mbox{\rm diag}(H^{p_1+2}-1,\ldots,H^{p_N+2}-1) \bigr)
 (\boldsymbol{\gamma}\eta_{(0)}'''+\delta^2\boldsymbol{\psi}'')\nonumber \\
&\quad
 -A_1\bigl\{ \delta^{-4}
 \bigl( \mbox{\rm diag}(H^{p_1}-1-\delta^2p_1\eta_{(0)},\ldots,H^{p_N}-1-\delta^2p_N\eta_{(0)}) \bigr)
 \boldsymbol{\gamma}\eta_{(0)}' \nonumber\\
&\qquad\qquad
 +\delta^{-2}\bigl( \mbox{\rm diag}(H^{p_1}-1,\ldots,H^{p_N}-1) \bigr)\boldsymbol{\psi} \bigr\}.
 \nonumber
\end{align}
To summarize, the Isobe--Kakinuma model \eqref{as:IK} is reduced to 
\begin{equation}\label{red:IK4}
\begin{cases}
 (\eta_{(0)}-2\gamma)\zeta + 2(\gamma-\eta_{(0)})\psi_0'
  + (\boldsymbol{1}-\boldsymbol{a}_0)\cdot\boldsymbol{\psi}' = g_0 + \delta^2 f_0, \\
(\boldsymbol{1}-\boldsymbol{a}_0)\psi_0'' - \delta^2(A_0-\boldsymbol{1}\otimes\boldsymbol{a}_0)\boldsymbol{\psi}''
 + A_1\boldsymbol{\psi} = \boldsymbol{g}+\delta^2\boldsymbol{f}, \\
 \zeta + \psi_0' = g_{N+1} + \delta^2 f_{N+1}
\end{cases}
\end{equation}
for the remainder terms $(\zeta,\psi_0,\boldsymbol{\psi})$. 
We note that $g_0$ and $g_{N+1}$ are polynomials in $(\eta_{(0)}, \eta_{(0)}',\eta_{(0)}'')$, 
that $\boldsymbol{g}$ is a polynomial in $(\eta_{(0)}, \eta_{(0)}',\eta_{(0)}''')$, and that 
$g_0,g_{N+1} \in B_e^\infty$ and $\boldsymbol{g} \in B_o^\infty$. 
Moreover, $f_0$ and $f_{N+1}$ are polynomials in 
$(\eta_{(0)}, \eta_{(0)}',\eta_{(0)}'', \zeta, \psi_0', \boldsymbol{\psi}, \boldsymbol{\psi}')$, 
and $\boldsymbol{f}$ is a polynomial in 
$(\eta_{(0)}, \eta_{(0)}',\eta_{(0)}''', \zeta, \psi_0'', \boldsymbol{\psi}, \delta^2\boldsymbol{\psi}'')$, 
whose coefficients are polynomials in $\delta$.

In view of \eqref{red:IK4} we proceed to consider the following system of linear ordinary differential equations 
\begin{equation}\label{red:le1}
\begin{cases}
 (\eta_{(0)}-2\gamma)\zeta + 2(\gamma-\eta_{(0)})\psi_0'
  + (\boldsymbol{1}-\boldsymbol{a}_0)\cdot\boldsymbol{\psi}' = F_0, \\
 (\boldsymbol{1}-\boldsymbol{a}_0)\psi_0'' - \delta^2(A_0-\boldsymbol{1}\otimes\boldsymbol{a}_0)\boldsymbol{\psi}''
 + A_1\boldsymbol{\psi} = \boldsymbol{F}, \\
 \zeta + \psi_0' = F_{N+1}
\end{cases}
\end{equation}
for unknowns $(\zeta,\psi_0,\boldsymbol{\psi})$. 
We can rewrite the third equation in \eqref{red:le1} as 
\begin{equation}\label{red:z}
\zeta = -\psi_0' + F_{N+1},
\end{equation}
which enable us to remove the unknown variable $\zeta$ from the equations. 
It follows from the second equations in \eqref{red:le1} that 
\[
\boldsymbol{\psi} = -\boldsymbol{\gamma}\psi_0'' + \delta^2A_2\boldsymbol{\psi}'' + A_1^{-1}\boldsymbol{F},
\]
where $A_2=A_1^{-1}(A_0-\boldsymbol{1}\otimes\boldsymbol{a}_0)$. 
Plugging this and \eqref{red:z} into the first equation in \eqref{red:le1}, we obtain 
\begin{equation}\label{red:psi0}
-\gamma\psi_0'''+(4\gamma-3\eta_{(0)})\psi_0' = F_0 + (2\gamma-\eta_{(0)})F_{N+1}
 - \boldsymbol{\gamma}\cdot\boldsymbol{F}'
  - \delta^2\boldsymbol{\beta}\cdot\boldsymbol{\psi}''',
\end{equation}
where $\boldsymbol{\beta}=A_2^{\rm T}(\boldsymbol{1}-\boldsymbol{a}_0)$. 
This is the equation for $\psi_0$. 
In order to derive equations to solve $\boldsymbol{\psi}$, we put $\chi=\psi_0''$. 
Then, it follows from \eqref{red:le1} that 
\begin{equation}\label{red:psi}
\begin{cases}
 (\boldsymbol{1}-\boldsymbol{a}_0)\cdot\boldsymbol{\psi}'
  = F_0 + (2\gamma - \eta_{(0)})F_{N+1} + (3\eta_{(0)} - 4\gamma)\psi_0', \\
 (\boldsymbol{1}-\boldsymbol{a}_0)\chi
 - \delta^2(A_0-\boldsymbol{1}\otimes\boldsymbol{a}_0)\boldsymbol{\psi}''
 + A_1\boldsymbol{\psi} = \boldsymbol{F},
\end{cases}
\end{equation}
This is the equations for $(\chi,\boldsymbol{\psi})$. 
Of course, $\chi$ is expected to be $\psi_0''$, but we do not know it a priori. 
However, we have the following lemma.

\begin{lemma}\label{red:lem0}
If $\zeta,\psi_0,\boldsymbol{\psi}$, and $\chi$ solve \eqref{red:z}--\eqref{red:psi} and satisfies 
$\chi(0)=\psi_0''(0)$, then we have $\chi=\psi_0''$. 
Particularly, $\zeta,\psi_0$, and $\boldsymbol{\psi}$ solve \eqref{red:le1}. 
\end{lemma}

\begin{proof}[{\bf Proof}.]
It follows from the second equation in \eqref{red:psi} that 
\[
\boldsymbol{\psi} = -\boldsymbol{\gamma}\chi + \delta^2A_2\boldsymbol{\psi}'' + A_1^{-1}\boldsymbol{F}.
\]
Plugging this into the first equation in \eqref{red:psi} we have 
\[
-\gamma\chi'+(4\gamma-3\eta_{(0)})\psi_0' = F_0 + (2\gamma-\eta_{(0)})F_{N+1}
 - \boldsymbol{\gamma}\cdot\boldsymbol{F}'
  - \delta^2\boldsymbol{\beta}\cdot\boldsymbol{\psi}'''.
\]
Comparing this with \eqref{red:psi0} we obtain $(\chi-\psi_0'')'=0$, which gives the desired results. 
\end{proof}

\section{Green's functions}
\label{section:gf}
In view of \eqref{red:psi0}, we first consider the solvability of the equation for an unknown $u$ 
\begin{equation}\label{gf:ode}
-\gamma u''+(4\gamma-3\eta_{(0)})u = f
\end{equation}
under the condition $u(x)\to0$ as $x\to\pm\infty$. 
We remind that $\eta_{(0)}(x) = 4\gamma\,\mbox{\rm sech}^2x$, so that 
this equation does not depend essentially on the positive constant $\gamma$. 
This equation has already been analyzed by K. O. Friedrichs and D. H. Hyers \cite{FriedrichsHyers1954}. 
Let us recall briefly their result. 
In order to construct Green's function, we consider the initial value problems with homogeneous equation: 
\[
\begin{cases}
-\gamma u_1''+(4\gamma-3\eta_{(0)})u_1 = 0, \\
u_1(0) = 0, \quad u_1'(0) = 1,
\end{cases}
\qquad
\begin{cases}
-\gamma u_2''+(4\gamma-3\eta_{(0)})u_2 = 0, \\
u_2(0) = 1, \quad u_2'(0) = 0.
\end{cases}
\]
The solutions of these initial value problems have the form 
\[
\begin{cases}
u_1(x) = \mbox{\rm sech}^2x \tanh x, \\
u_2(x) = \frac18 ( - 6 - \cosh(2x) + 15\mbox{\rm sech}^2x - 15 x \mbox{\rm sech}^2x \tanh x).
\end{cases}
\]
We note that these fundamental solutions satisfy 
\begin{equation}\label{estimate of fundamental solution}
|u_1^{(k)}(x)| \leq C_k \mbox{e}^{-2|x|}, \qquad
|u_2^{(k)}(x)| \leq C_k \mbox{e}^{2|x|}
\end{equation}
for $k=0,1,2,\ldots$. 
Since the Wronskian is $u_1'(x)u_2(x) - u_1(x)u_2'(x) \equiv 1$, 
the solution $u$ of \eqref{gf:ode} can be written as 
\[
u(x) = \gamma^{-1}\biggl(C_1 - \int_0^x u_2(y) f(y) {\rm d}y \biggr) u_1(x)
 + \gamma^{-1}\biggl(C_2 + \int_0^x u_1(y) f(y) {\rm d}y \biggr) u_2(x),
\]
where $C_1$ and $C_2$ are arbitrary constants. 
In order that this solution satisfies the condition $u(x) \to 0$ as $x\to\pm\infty$, 
as necessary conditions, we obtain 
\[
C_2 + \int_0^{\infty} u_1(y) f(y) {\rm d}y = C_2 - \int_{-\infty}^0 u_1(y) f(y) {\rm d}y = 0,
\]
so that we have a necessary condition
\begin{equation}\label{gf:nc}
\int_{-\infty}^{\infty} u_1(y) f(y) {\rm d}y = 0
\end{equation}
for the existence of the solution, and that the solution has the form 
\[
u(x) = \gamma^{-1} u_1(x) \biggl( C_1 - \int_0^x u_2(y) f(y) {\rm d}y \biggr)
 - \gamma^{-1} u_2(x) \int_x^{\infty} u_1(y) f(y) {\rm d}y.
\]
Nonuniqueness of the solution comes from the translation invariance of the original equations.

\begin{proposition}\label{gf:prop1}
Let $\gamma>0$. 
For any $f\in B_e^\infty$, there exists a unique solution $u \in B_e^\infty$ to \eqref{gf:ode}. 
Moreover, for any $k=0,1,2,\ldots$, the solution satisfies 
\[
\|u\|_{k+2} \leq C_0\|f\|_k + C_k\|f\|_0,
\]
where $C_k$ is a positive constant depending on $k$ while $C_0$ does not. 
\end{proposition}

\begin{proof}[{\bf Proof}.]
Since $f$ is even and $u_1$ is odd, the necessary condition \eqref{gf:nc} for the existence of the 
solution is automatically satisfied. 
Moreover, we restricted ourselves to a class of solutions which are even so that the solution is unique and given by 
\[
u(x) = -\gamma^{-1}u_1(x) \int_0^x u_2(y) f(y) {\rm d}y - \gamma^{-1}u_2(x) \int_x^{\infty} u_1(y) f(y) {\rm d}y.
\]
Using this solution formula, we easily obtain $\|u\|_{k+2} \leq C_k\|f\|_k$. 

We proceed to improve this estimate. 
Differentiating \eqref{gf:ode} $j$-times, we have 
\[
\textstyle
-\gamma u^{(j) \prime\prime}+(4\gamma-3\eta_{(0)})u^{(j)} = f^{(j)}+3[(\frac{\rm d}{{\rm d}x})^j,\eta_{(0)}]u
\]
Applying the previous estimate with $k=0$ to $u^{(j)}$ and adding them for $j=1,2,\ldots,k$, we obtain 
$\|u\|_{k+2} \leq C\|f\|_k + C_k\|u\|_{W^{k,\infty}}$, which together with the interpolation inequality 
$\|u\|_{W^{k,\infty}} \leq \epsilon\|u\|_{W^{k+2,\infty}} + C_\epsilon\|u\|_{L^\infty}
 \leq \epsilon\|u\|_{k+2}+C_\epsilon\|u\|_2$ for any $\epsilon>0$ yields the refined estimate. 
\end{proof}

In view of \eqref{red:psi}, we then consider the solvability of the system of ordinary differential equations 
with constant coefficients 
\begin{equation}\label{gf:varphi}
\begin{cases}
 (\boldsymbol{1}-\boldsymbol{a}_0)\cdot\boldsymbol{\varphi} = f_0, \\
 (\boldsymbol{1}-\boldsymbol{a}_0)\varphi_0
  - \delta^2(A_0-\boldsymbol{1}\otimes\boldsymbol{a}_0)\boldsymbol{\varphi}''
  + A_1\boldsymbol{\varphi} = \boldsymbol{f},
\end{cases}
\end{equation}
for unknowns $\varphi_0$ and $\boldsymbol{\varphi}=(\varphi_1,\ldots,\varphi_N)^{\rm T}$ 
while $f_0$ and $\boldsymbol{f}=(f_1,\ldots,f_N)^{\rm T}$ are given functions. 
We procced to construct Green's function to this system. 
By taking the Fourier transform of \eqref{gf:varphi}, we obtain
\[
\begin{pmatrix}
 0 & (\boldsymbol{1}-\boldsymbol{a}_0)^{\rm T} \\
 \boldsymbol{1}-\boldsymbol{a}_0 & (\delta\xi)^2(A_0 - \boldsymbol{1}\otimes\boldsymbol{a}_0) + A_1
\end{pmatrix}
\begin{pmatrix}
 \hat{\varphi}_0 \\ \hat{\mbox{\boldmath$\varphi$}}
\end{pmatrix}
=
\begin{pmatrix}
 \hat{f}_0 \\ \hat{\mbox{\boldmath$f$}}
\end{pmatrix}.
\]
Now, we need to show that the coefficient matrix is invertible. 
Put 
\begin{equation}\label{gf:defq}
\mathfrak{q}(\xi^2) = -\det
\begin{pmatrix}
 0 & (\boldsymbol{1}-\boldsymbol{a}_0)^{\rm T} \\
 \boldsymbol{1}-\boldsymbol{a}_0 & \xi^2(A_0 - \boldsymbol{1}\otimes\boldsymbol{a}_0) + A_1
\end{pmatrix}.
\end{equation}
Then, we have the following lemma.

\begin{lemma}
$\mathfrak{q}(\xi^2)$ is a polynomial in $\xi^2$ of degree $N-1$. 
Moreover, there exists a positive constant $c_0$ such that for any $\xi\in\mathbf{R}$ 
we have $\mathfrak{q}(\xi^2) \geq c_0>0$. 
\end{lemma}

\begin{proof}[{\bf Proof}.]
It is easy to see that 
\begin{align*}
\mathfrak{q}(\xi^2)
&= -\det
\begin{pmatrix}
 0 & (\boldsymbol{1}-\boldsymbol{a}_0)^{\rm T} \\
 \boldsymbol{1}-\boldsymbol{a}_0 & \xi^2(A_0 - \boldsymbol{a}_0\otimes\boldsymbol{a}_0) + A_1
\end{pmatrix}.
\end{align*}
Here, we will show that the symmetric matrix $A_0 - \boldsymbol{a}_0\otimes\boldsymbol{a}_0$ is positive. 
In fact, for any $\boldsymbol{\phi}=(\phi_1,\ldots,\phi_N)^{\rm T}\in\mathbf{R}^N$, we see that 
\begin{align*}
\boldsymbol{\phi}\cdot(A_0 - \boldsymbol{a}_0\otimes\boldsymbol{a}_0)\boldsymbol{\phi}
&= \sum_{i,j=1}^N\frac{1}{p_i+p_j+1}\phi_i\phi_j - \biggl(\sum_{j=1}^N\frac{1}{p_j+1}\phi_j\biggr)^2 \\
&= \int_0^1\biggl(\sum_{j=1}^N\phi_jz^{p_j}\biggr)^2{\rm d}z
 - \biggl(\int_0^1\sum_{j=1}^N\phi_jz^{p_j}{\rm d}z\biggr)^2 \\
&\geq 0
\end{align*}
by the Cauchy--Schwarz inequality. 
Moreover, the equality holds if and only if $\sum_{j=1}^N\phi_jz^{p_j}$ is constant for $z\in[0,1]$, that is, 
$\boldsymbol{\phi}=\boldsymbol{0}$. 
This shows that $A_0 - \boldsymbol{a}_0\otimes\boldsymbol{a}_0$ is positive.

Next, we note that for any invertible matrix $A$ and any vector $\boldsymbol{x}$ we have 
\[
\det
\begin{pmatrix}
 0 & \boldsymbol{x}^{\rm T} \\
 \boldsymbol{x} & A
\end{pmatrix}
= -(\det A)\boldsymbol{x}\cdot A^{-1}\boldsymbol{x}.
\]
This equality comes easily from the identity 
\[
\begin{pmatrix}
 1 & -\boldsymbol{x}^{\rm T}A^{-1} \\
 \boldsymbol{0} & \mbox{Id}
\end{pmatrix}
\begin{pmatrix}
 0 & \boldsymbol{x}^{\rm T} \\
 \boldsymbol{x} & A
\end{pmatrix}
=
\begin{pmatrix}
 -\boldsymbol{x}\cdot A^{-1}\boldsymbol{x} & \boldsymbol{0}^{\rm T} \\
 \boldsymbol{x} & A
\end{pmatrix}.
\]
Since $A_1$ is also positive, for each $\xi\in\mathbf{R}$ 
the matrix $\xi^2(A_0 - \boldsymbol{a}_0\otimes\boldsymbol{a}_0) + A_1$ is positive too. 
Therefore, we have $\mathfrak{q}(\xi^2)>0$, which yields the second assertion of the lemma.

It is easy to see that $\mathfrak{q}(\xi^2)$ is a polynomial in $\xi^2$ of degree less than or equal to 
$N-1$ and that the coefficient of $\xi^{2(N-1)}$ is 
\[
-
\begin{pmatrix}
 0 & (\boldsymbol{1}-\boldsymbol{a}_0)^{\rm T} \\
 \boldsymbol{1}-\boldsymbol{a}_0 & A_0 - \boldsymbol{a}_0\otimes\boldsymbol{a}_0
\end{pmatrix},
\]
which is positive due to the positivity $A_0 - \boldsymbol{a}_0\otimes\boldsymbol{a}_0$. 
Therefore, we obtain the first assertion of the lemma. 
\end{proof}

We denote the inverse matrix by 
\[
\begin{pmatrix}
 0 & (\boldsymbol{1}-\boldsymbol{a}_0)^{\rm T} \\
 \boldsymbol{1}-\boldsymbol{a}_0 & (\delta\xi)^2(A_0 - \boldsymbol{1}\otimes\boldsymbol{a}_0) + A_1
\end{pmatrix}^{-1}
=
\begin{pmatrix}
 q((\delta\xi)^2) & \boldsymbol{q}((\delta\xi)^2)^{\rm T} \\
 \boldsymbol{q}((\delta\xi)^2) & Q((\delta\xi)^2)
\end{pmatrix},
\]
where $\boldsymbol{q}(\xi^2) = (q_1(\xi^2),\ldots,q_N(\xi^2))^{\rm T}$ and 
$Q(\xi^2)=(q_{ij}(\xi^2))_{1\leq i,j\leq N}$. 
Then, the solution $(\varphi_0,\boldsymbol{\varphi})$ to \eqref{gf:varphi} can be expressed as 
\begin{equation}\label{gf:solf1}
\begin{cases}
 \hat{\varphi}_0 = q((\delta\xi)^2)\hat{f}_0 + \boldsymbol{q}((\delta\xi)^2)\cdot\hat{\boldsymbol{f}}, \\
 \hat{\boldsymbol{\varphi}} = \boldsymbol{q}((\delta\xi)^2)\hat{f}_0 + Q((\delta\xi)^2)\hat{\boldsymbol{f}}.
\end{cases}
\end{equation}
We note that $q(\xi^2), q_i(\xi^2)$, and $q_{ij}(\xi^2)$ are rational functions in $\xi^2$, more precisely, 
we can express them as 
\[
q(\xi^2) = q_1\xi^2 + q_0 + \frac{\mathfrak{p}(\xi^2)}{\mathfrak{q}(\xi^2)}, \qquad
q_i(\xi^2) = q_{0i} + \frac{\mathfrak{p}_i(\xi^2)}{\mathfrak{q}(\xi^2)}, \qquad
q_{ij}(\xi^2) = \frac{\mathfrak{p}_{ij}(\xi^2)}{\mathfrak{q}(\xi^2)},
\]
for $i,j=1,\ldots,N$, where $\mathfrak{q}(\xi^2)$ is the polynomial in $\xi^2$ of degree $N-1$ 
defined by \eqref{gf:defq}. 
Moreover, $\mathfrak{p}(\xi^2)$, $\mathfrak{p}_i(\xi^2)$, and $\mathfrak{p}_{ij}(\xi^2)$ are polynomials 
in $\xi^2$ of degree less than $N-1$. 
Therefore, by taking the inverse Fourier transform of \eqref{gf:solf1} we obtain 
\begin{equation}\label{gf:solf2}
\begin{cases}
 \displaystyle
 \varphi_0 = q_1\delta^2f_0'' + q_0f_0 + r_{\delta}*f_0 + \sum_{j=1}^N r_{j\delta}*f_j, \\
 \boldsymbol{\varphi} = \boldsymbol{q}_0f_0
  + \boldsymbol{r}_{\delta}*f_0 + R_{\delta}*\boldsymbol{f},
\end{cases}
\end{equation}
where $\boldsymbol{q}_0 = (q_{01},\ldots,q_{0N})^{\rm T}$, 
$\mbox{\boldmath$r$}(x) = (r_1(x),\ldots,r_N(x))^{\rm T}$, 
$R(x) = \bigl( r_{ij}(x) \bigr)_{1\leq i,j\leq N}$, and 
$r(x)$, $r_j(x)$, $r_{ij}(x)$ are Fourier inverse of 
$\frac{\mathfrak{p}(\xi^2)}{\mathfrak{q}(\xi^2)}$, 
$\frac{\mathfrak{p}_j(\xi^2)}{\mathfrak{q}(\xi^2)}$, 
$\frac{\mathfrak{p}_{ij}(\xi^2)}{\mathfrak{q}(\xi^2)}$, respectively. 
We also used the notation $f_{\delta}(x)=\delta^{-1}f(\delta^{-1}x)$ for any function $f(x)$.

In view of these solution formulae, we consider a function $r(x)$ defined by 
\begin{equation}\label{gf:fr}
r(x) = \frac{1}{2\pi}\int_{\mathbf{R}} \frac{\mathfrak{p}(\xi^2)}{\mathfrak{q}(\xi^2)}
  \mbox{\rm e}^{\mbox{\rm\scriptsize i}x\xi} {\rm d}\xi,
\end{equation}
where $\mathfrak{q}(\xi^2)$ and $\mathfrak{p}(\xi^2)$ are polynomials in $\xi^2$ of degree $N$ and 
of degree less than $N-1$, respectively, and $\mathfrak{q}(\xi^2)\geq c_0>0$. 
It is easy to see that $r(x)$ is a real valued, even, and continuous function on $\mathbf{R}$. 
Since $\mathfrak{q}(\xi^2)$ is a polynomials in $\xi^2$ of degree $N-1$ with real coefficients 
and positive definite, roots of $\mathfrak{q}(\xi^2)=0$ have the form 
\[
\xi = \alpha_k \pm {\rm i}\beta_k, \quad k=1,2,\ldots,N-1
\]
with $0 < \beta_1 \leq \beta_2 \leq \cdots \leq \beta_{N-1}$. 
Therefore, by the residue theorem we have an expression 
\[
r(x) = \sum_{k=1}^{N-1} r_k^\pm {\rm e}^{{\rm i}\alpha_kx - \beta_k|x|} \quad\mbox{for}\quad 
 x \gtrless 0
\]
with some complex constants $r_1^\pm,\ldots,r_{N-1}^\pm$. 
By taking into account the continuity at $x=0$, we have also 
\[
r'(x) = \sum_{k=1}^{N-1} r_k^\pm ( i\alpha_k \mp \beta_k )
 {\rm e}^{{\rm i}\alpha_kx - \beta_k|x|} \quad\mbox{for}\quad  x \gtrless 0.
\]
We note that the second derivative $r''(x)$ contains in general the Dirac delta function. 
Thanks of these expressions, we have the following lemma.

\begin{lemma}\label{gf:lem1}
Let $r(x)$ be defined by \eqref{gf:fr}. 
Then, $r(x)$ is a real value even function and $r \in W^{1,\infty}(\mathbb{R})$. 
Moreover, there exists positive constants $\beta_1$ and $C$ such that for any $x\in\mathbf{R}$ 
we have a pointwise estimate
\[
|r(x)| + |r'(x)| \leq C{\rm e}^{-\beta_1|x|}. 
\]
\end{lemma}

For $\delta>0$, we put $r_{\delta}(x) = \delta^{-1}r(\delta^{-1}x)$ and consider the function $r_{\delta}*f$. 
We remind the function spaces $B_e^k$ and $B_o^k$ for $k=0,1,\ldots$ defined in Section \ref{section:intro} 
and that the norm $\|\cdot\|$ was defined by \eqref{intro:norm}.

\begin{lemma}\label{gf:lem2}
Let $r(x)$ be defined by \eqref{gf:fr} and fix $\delta_0$ such that $0<\delta_0<\beta_1$. 
Then, there exists a constant $C$ such that for any $f\in B_\alpha^0$ $(\alpha=e\;\mbox{or}\;o)$ and 
$\delta \in(0,\delta_0]$ we have $r_{\delta}*f \in B_\alpha^1$. 
Moreover, it holds that
\[
\|r_{\delta}*f\| + \delta\|(r_{\delta}*f)'\| \leq C\|f\|.
\]
\end{lemma}

\begin{proof}[{\bf Proof}.]
It follows from Lemma \ref{gf:lem1} that 
\begin{align*}
|(r_{\delta}*f)(x)| 
&\leq \frac{1}{\delta}\int_{\mathbf{R}}
 \biggl|r\biggl(\frac{x-y}{\delta}\biggr)\biggr||f(y)|{\rm d}y \\
&\leq \frac{C}{\delta}\|f\|\int_{\mathbf{R}}
 {\rm e}^{-\frac{\beta_1}{\delta}|x-y|-|y|}{\rm d}y \\
&= C\|f\|\biggl\{ \frac{1}{\beta_1+\delta}
  \bigl( {\rm e}^{-|x|} + {\rm e}^{-\frac{\beta_1}{\delta}|x|} \bigr)
 + \frac{1}{\beta_1-\delta}
  \bigl( {\rm e}^{-|x|} - {\rm e}^{-\frac{\beta_1}{\delta}|x|} \bigr) \biggr\} \\
&\leq \frac{4C}{\beta_1^2 - \delta_0^2}\|f\| {\rm e}^{-|x|},
\end{align*}
which implies $\|r_{\delta}*f\| \leq \frac{4C}{\beta_1^2 - \delta_0^2}\|f\|$. 
In view of $(r_{\delta}*f)'=r_{\delta}'*f$, a similar estimate holds for 
$(r_{\delta}*f)'$. 
Moreover, it is easy to see that if $f$ is even or odd, then so is $r_{\delta}*f$, respectively. 
Therefore, we obtain the desired result. 
\end{proof}

In order to give an estimate for the solution $(\varphi_0,\boldsymbol{\varphi})$ to \eqref{gf:varphi}, 
it is convenient to introduce the following weighted norm 
\begin{equation}\label{gf:wnorm}
\|u\|_{k+2,\delta} = \|u\|_k + \delta\|u\|_{k+1} + \delta^2\|u\|_{k+2}
\end{equation}
for $k=0,1,2,\ldots$. 
By the above arguments, we obtain the following proposition.

\begin{proposition}\label{gf:prop2}
There exist positive constants $\delta_0$ and $C$ such that 
for any $f_0, \boldsymbol{f} \in B_\alpha^\infty$ $(\alpha = e$ or $o)$ and any $\delta\in(0,\delta_0]$, 
there exists a unique solution $\varphi_0,\boldsymbol{\varphi} \in B_\alpha^\infty$ to \eqref{gf:varphi}. 
Moreover, for any $k=0,1,2,\ldots$, the solution satisfies 
\[
\begin{cases}
\|\varphi_0\|_k + \|\boldsymbol{\varphi}\|_{k+2,\delta}
 \leq C( \|f_0\|_{k+2,\delta} + \|\boldsymbol{f}\|_k ), \\
\delta^2\|\boldsymbol{\varphi}''\|_k \leq C( \delta^2\|f_0''\|_k + \|\boldsymbol{f}\|_k ).
\end{cases}
\]
\end{proposition}

\section{Existence of small amplitude solitary waves}
\label{section:ET}
We begin to give an existence theorem of the solution $(\zeta,\psi_0,\boldsymbol{\psi})$ 
to the system of linear ordinary differential equations \eqref{red:le1}.

\begin{proposition}\label{ET:prop1}
There exists a positive constant $\delta_0$ such that for any $F_0,F_{N+1} \in B_e^\infty$, 
any $\boldsymbol{F} \in B_o^\infty$, and any $\delta\in(0,\delta_0]$, 
there exists a unique solution $(\zeta,\psi_0,\boldsymbol{\psi})$ to \eqref{red:le1} satisfying 
$\zeta,\psi_0' \in B_e^\infty$, $\boldsymbol{\psi} \in B_o^\infty$, and $\psi_0(0)=0$. 
Moreover, for any $k=0,1,2,\ldots$, the solution satisfies 
\begin{align*}
\|\psi_0'\|_{k+2} + \|\zeta\|_{k+2,\delta} + \|\boldsymbol{\psi}\|_{k+3,\delta}
&\leq C_0( \|(F_0,F_{N+1})\|_{k+2,\delta} + \|\boldsymbol{F}\|_{k+1} ) \\
&\quad + C_k(  \|(F_0,F_{N+1})\|_{2,\delta} + \|\boldsymbol{F}\|_{1} ),
\end{align*}
where $C_k$ is a positive constant depending on $k$ while $C_0$ does not. 
\end{proposition}

\begin{proof}[{\bf Proof.}]
Thanks of Lemma \ref{red:lem0}, it is sufficient to solve \eqref{red:psi0}--\eqref{red:psi} for 
$(\psi_0,\boldsymbol{\psi},\chi)$ and then to define $\zeta$ by \eqref{red:z}. 
To solve \eqref{red:psi0}--\eqref{red:psi}, we use the standard iteration argument by applying 
the existence theorems, that is, Propositions \ref{gf:prop1} and \ref{gf:prop2} given in the previous section. 
Suppose that $\tilde{\psi}_0$ is given so that $\tilde{\psi}_0' \in B_e^\infty$. 
By Proposition \ref{gf:prop2} under the condition $0<\delta\leq \delta_0$ 
there exists a unique solution $(\varphi_0,\boldsymbol{\varphi}) \in B_e^\infty$ to 
\[
\begin{cases}
 (\boldsymbol{1}-\boldsymbol{a}_0)\cdot\boldsymbol{\varphi}
  = F_0 + (2\gamma - \eta_{(0)})F_{N+1} + (3\eta_{(0)} - 4\gamma)\tilde{\psi}_0', \\
 (\boldsymbol{1}-\boldsymbol{a}_0)\varphi_0
 - \delta^2(A_0-\boldsymbol{1}\otimes\boldsymbol{a}_0)\boldsymbol{\varphi}''
 + A_1\boldsymbol{\varphi} = \boldsymbol{F}'.
\end{cases}
\]
Then, we define $\chi$ and $\boldsymbol{\psi}$ by 
\[
\chi(x)=\int_0^x\varphi_0(y){\rm d}y, \qquad
\boldsymbol{\psi}(x)=\int_0^x\boldsymbol{\varphi}(y){\rm d}y,
\]
which are odd functions and satisfy 
\begin{equation}\label{exist:appeq1}
\begin{cases}
 (\boldsymbol{1}-\boldsymbol{a}_0)\cdot\boldsymbol{\psi}'
  = F_0 + (2\gamma - \eta_{(0)})F_{N+1} + (3\eta_{(0)} - 4\gamma)\tilde{\psi}', \\
 (\boldsymbol{1}-\boldsymbol{a}_0)\chi
 - \delta^2(A_0-\boldsymbol{1}\otimes\boldsymbol{a}_0)\boldsymbol{\psi}''
 + A_1\boldsymbol{\psi} = \boldsymbol{F}.
\end{cases}
\end{equation}
Then, by Proposition \ref{gf:prop1} there exists a unique solution $\psi_0$ to 
\begin{equation}\label{exist:appeq2}
-\gamma\psi_0'''+(4\gamma-3\eta_{(0)})\psi_0' = F_0 + (2\gamma-\eta_{(0)})F_{N+1}
 + \boldsymbol{\gamma}\cdot\boldsymbol{F}'
  - \delta^2\boldsymbol{\beta}\cdot\boldsymbol{\psi}'''
\end{equation}
satisfying $\psi_0' \in B_e^\infty$ and $\psi_0(0)=0$. 
By using these solutions $(\chi,\boldsymbol{\psi},\psi_0)$, we define maps 
$\mathcal{T}_0:\tilde{\psi}_0\mapsto\psi_0$, $\mathcal{T}_1:\tilde{\psi}_0\mapsto\chi$, 
$\boldsymbol{\mathcal{T}}_2:\tilde{\psi}_0\mapsto\boldsymbol{\psi}$. 
Clearly, $\mathcal{T}_0$ maps 
\[
X=\{ \psi_0 \in C^\infty(\mathbf{R}) \,|\, \psi_0' \in B_e^\infty,\psi_0(0)=0\}
\]
into itself. 
We will show that $\mathcal{T}_0$ is a contraction map with respect to an appropriate norm. 
To this end, let $\tilde{\psi}_0 \in X$ and let $(\chi,\boldsymbol{\psi},\psi_0)$ be the solutions as above. 
By Proposition \ref{gf:prop2} we have 
\begin{align}\label{exist:est1}
\|\chi'\|_k + \|\boldsymbol{\psi}'\|_{k+2,\delta}
\leq C( \|(F_0,F_{N+1})\|_{k+2,\delta} + \|\boldsymbol{F}\|_{k+1} + \|\tilde{\psi}_0'\|_{k+2} ) 
 + C_k( \|F_{N+1}\| + \|\tilde{\psi}_0'\| ), \\
\label{exist:est2}
\delta^2\|\boldsymbol{\psi}'''\|_k
\leq C\|(F_0,F_{N+1})\|_{k+2,\delta} + \|\boldsymbol{F}\|_{k+1} + \delta^2\|\tilde{\psi}_0'\|_{k+2} ) 
 + C_k( \|F_{N+1}\| + \delta^2\|\tilde{\psi}_0'\| ). 
\end{align}
By Proposition \ref{gf:prop1} we have 
\[
\|\psi_0'\|_{k+2}
\leq C( \|(F_0,F_{N+1})\|_k + \|\boldsymbol{F}\|_k + \delta^2\|\boldsymbol{\psi}'''\|_k )
 + C_k( \|(F_0,F_{N+1})\| + \|\boldsymbol{F}\| + \delta^2\|\boldsymbol{\psi}'''\| ),
\]
which together with \eqref{exist:est2} implies 
\[
\|\psi_0'\|_{k+2}
\leq C( \|(F_0,F_{N+1})\|_{k+2,\delta} + \|\boldsymbol{F}\|_{k+1} + \delta^2\|\tilde{\psi}_0'\|_{k+2} )
  + C_k( \|(F_0,F_{N+1})\| + \|\boldsymbol{F}\| + \delta^2\|\tilde{\psi}_0'\| ). 
\]
Particularly, we obtain 
\[
\|\psi_0'\|_2
\leq C( \|(F_0,F_{N+1})\|_{2,\delta} + \|\boldsymbol{F}\|_{1} + \delta^2\|\tilde{\psi}_0'\|_{2} ). 
\]
Therefore, by taking $\delta_0$ so small that $2C\delta_0^2 \leq 1$, we obtain 
\[
\|\psi_0'\|_2
\leq \frac12\|\tilde{\psi}_0'\|_2 + C( \|(F_0,F_{N+1})\|_{2,\delta} + \|\boldsymbol{F}\|_1 )
\]
and 
\begin{align}\label{exist:est3}
\|\psi_0'\|_{k+2}
&\leq \frac12\|\tilde{\psi}_0'\|_{k+2} + C( \|(F_0,F_{N+1})\|_{k+2,\delta} + \|\boldsymbol{F}\|_{k+1} ) \\
&\quad\;
  + C_k\|\tilde{\psi}_0'\|_2
  + C_k( \|(F_0,F_{N+1})\|_{2,\delta} + \|\boldsymbol{F}\|_1 ).
  \nonumber
\end{align}
In exactly the same way as above, we obtain also 
\[
\begin{cases}
 \|\mathcal{T}_0(\tilde{\psi}_0)' - \mathcal{T}_0(\psi_0)'\|_2 \leq \frac12 \|\psi_0'-\tilde{\psi}_0'\|_2, \\
 \|\mathcal{T}_0(\tilde{\psi}_0)' - \mathcal{T}_0(\psi_0)'\|_{k+2} \leq \frac12 \|\psi_0'-\tilde{\psi}_0'\|_{k+2}
  + C_k\|\psi_0'-\tilde{\psi}_0'\|_2, \\
 \|\mathcal{T}_1(\tilde{\psi}_0)' - \mathcal{T}_1(\psi_0)'\|_k
  + \|\boldsymbol{\mathcal{T}}_2(\tilde{\psi}_0)' - \boldsymbol{\mathcal{T}}_2(\psi_0)'\|_{k+2,\delta}
  \leq C_k\|\psi_0'-\tilde{\psi}_0'\|_{k+2}.
\end{cases}
\]
On the other hand, it is easy to see that for any function $u$ satisfying $u(0)=0$ we have 
$\|u\|_{L^\infty} \leq \|u'\|$. 
Therefore, by the contraction mapping principle, the map $\mathcal{T}_0$ has a unique fixed point 
$\psi_0 \in X$. 
Using this fixed point $\psi_0$, we put $\chi=\mathcal{T}_2(\psi_0)$ and 
$\boldsymbol{\psi}=\boldsymbol{\mathcal{T}}_2(\psi_0)$. 
Then, we see easily that $(\psi_0,\chi,\boldsymbol{\psi})$ is a solution to \eqref{red:psi0}--\eqref{red:psi}. 
Moreover, it follows from \eqref{exist:est3} and \eqref{exist:est1} that 
\begin{align*}
\|\psi_0'\|_{k+2} + \|\boldsymbol{\psi}'\|_{k+2,\delta}
\leq C'( \|(F_0,F_{N+1})\|_{k+2,\delta} + \|\boldsymbol{F}\|_{k+1} )
  + C_k'( \|(F_0,F_{N+1})\|_{2,\delta} + \|\boldsymbol{F}\|_1 ),
\end{align*}
where $C'$ is a positive constant independent of $k$ whereas the constant $C_k'$ depends on $k$. 
It follows from Lemma \ref{red:lem0} that $\chi=\psi_0''$ so that 
\[
\boldsymbol{\psi} = A_1^{-1}( \boldsymbol{F} - (\boldsymbol{1}-\boldsymbol{a}_0)\psi_0''
 + \delta^2(A_0 - \boldsymbol{1}\otimes\boldsymbol{a}_0)\boldsymbol{\psi}''),
\]
which implies that 
$\|\boldsymbol{\psi}\| \leq C( \|\boldsymbol{F}\| + \|\psi_0'\|_2 + \|\boldsymbol{\psi}'\|_{2,\delta} )$. 
We note that 
$\|\boldsymbol{\psi}\|_{k+3,\delta} \leq \|\boldsymbol{\psi}'\|_{k+2,\delta} + 3\|\boldsymbol{\psi}\|$. 
Therefore, by defining $\zeta$ by \eqref{red:z} we see that $(\zeta,\psi_0,\boldsymbol{\psi})$ is the 
solution to \eqref{red:le1} satisfying the desired estimate. 
Uniqueness of the solution can be shown as in the above calculation under the restriction 
$0<\delta\leq\delta_0$. 
\end{proof}

In order to prove one of our main result in this paper, that is, Theorem \ref{intro:theorem}, 
it is sufficient to show an existence of the solution $(\zeta,\psi_0,\boldsymbol{\psi})$ to \eqref{red:IK4} 
together with a uniform bound of the solution with respect to the small parameter $\delta$. 
To this end, we need to give estimates of remainder terms $(f_0,\boldsymbol{f},f_{N+1})$ together with 
$(g_0,\boldsymbol{g},g_{N+1})$ in \eqref{red:IK4}. 
It is not difficult to show the following lemma.

\begin{lemma}\label{exist:lem1}
Suppose that $F(\boldsymbol{u})$ is a polynomial of $\boldsymbol{u}$ such that $F(\boldsymbol{0})=0$. 
Then, for any $k=2,3,4,\ldots$ and any $\delta\in(0,1]$ we have 
\[
\begin{cases}
 \|F(\boldsymbol{u})\|_{k+2,\delta} \leq C(\|\boldsymbol{u}\|)\|\boldsymbol{u}\|_{k+2,\delta}
  + C(k,\|\boldsymbol{u}\|_{k+1,\delta}), \\
 \|F(\boldsymbol{u})\|_{k+1} \leq C(\|\boldsymbol{u}\|)\|\boldsymbol{u}\|_{k+1}
  + C(k,\|\boldsymbol{u}\|)\|\boldsymbol{u}\|_1\|\boldsymbol{u}\|_k
  + C(k,\|\boldsymbol{u}\|_{k-1}).
\end{cases}
\]
\end{lemma}

We note again that $g_0$ and $g_{N+1}$ are polynomials in $(\eta_{(0)}, \eta_{(0)}',\eta_{(0)}'')$, 
that $\boldsymbol{g}$ is a polynomial in $(\eta_{(0)}, \eta_{(0)}',\eta_{(0)}''')$, and that 
$g_0,g_{N+1} \in B_e^\infty$ and $\boldsymbol{g} \in B_o^\infty$. 
Moreover, $f_0$ and $f_{N+1}$ are polynomials in 
$(\eta_{(0)}, \eta_{(0)}',\eta_{(0)}'', \zeta, \psi_0', \boldsymbol{\psi}, \boldsymbol{\psi}')$, 
and $\boldsymbol{f}$ is a polynomial in 
$(\eta_{(0)}, \eta_{(0)}',\eta_{(0)}''', \zeta, \psi_0'', \boldsymbol{\psi}, \delta^2\boldsymbol{\psi}'')$, 
whose coefficients are polynomials in $\delta$. 
Therefore, applying Lemma \ref{exist:lem1} to $(f_0,\boldsymbol{f},f_{N+1})$ we obtain the following lemma.

\begin{lemma}\label{exist:lem2}
Suppose that $\psi',\zeta \in B_e^\infty$, $\boldsymbol{\psi} \in B_o^\infty$, and $\delta\in(0,1]$ satisfy 
\[
\|\psi_0'\|_{k+2} + \|\zeta\|_{k+2,\delta} + \|\boldsymbol{\psi}\|_{k+3,\delta} \leq M_k
\]
for $k=2,3,4,\ldots$. 
Then, it holds that $f_0,f_{N+1} \in B_e^\infty$, $\boldsymbol{f} \in B_o^\infty$, and that 
for $k=3,4,5,\ldots$ we have 
\[
\begin{cases}
 \|(f_0,f_{N+1})\|_{4,\delta} + \delta\|\boldsymbol{f}\|_3 \leq C(M_2), \\
 \|(f_0,f_{N+1})\|_{k+2,\delta} + \delta\|\boldsymbol{f}\|_{k+1}
  \leq C(M_2)M_k + C(k,M_{k-1}).
\end{cases}
\]
\end{lemma}

Now, we are ready to prove Theorem \ref{intro:theorem}. 
We will construct the solution $(\zeta,\psi_0,\boldsymbol{\psi})$ to \eqref{red:IK4} by the standard 
iteration arguments. 
To this end, we introduce a function space 
\begin{align*}
\boldsymbol{X}=\{(\zeta,\psi_0,\boldsymbol{\psi})\in C^\infty(\mathbf{R}) \,|\, &
 \zeta,\psi_0'\in B_e^\infty, \boldsymbol{\psi}\in B_0^\infty, \psi_0(0)=0, \\ &
 \opnorm{(\zeta,\psi_0,\boldsymbol{\psi})}_{k,\delta} \leq M_k \;\mbox{for}\;k=1,2,3,\ldots \},
\end{align*}
where 
\[
\opnorm{(\zeta,\psi_0,\boldsymbol{\psi})}_{k,\delta} = \|\psi_0'\|_{k+2}+\|\zeta\|_{k+2,\delta}+\|\boldsymbol{\psi}\|_{k+3,\delta}
\]
and the constants $M_k$ for $k=1,2,3,\ldots$ will be defined later. 
Suppose that $(\tilde{\zeta},\tilde{\psi}_0,\tilde{\boldsymbol{\psi}}) \in \boldsymbol{X}$. 
Then, by Lemma \ref{exist:lem2} and Proposition \ref{ET:prop1}, under the condition $0<\delta\leq\delta_0$ 
there exists a unique solution $(\zeta,\psi_0,\boldsymbol{\psi})$ to 
\begin{equation}\label{ET:appeq}
\begin{cases}
 (\eta_{(0)}-2\gamma)\zeta + 2(\gamma-\eta_{(0)})\psi_0'
  + (\boldsymbol{1}-\boldsymbol{a}_0)\cdot\boldsymbol{\psi}' = g_0 + \delta^2 \tilde{f}_0, \\
(\boldsymbol{1}-\boldsymbol{a}_0)\psi_0''
 - \delta^2(A_0-\boldsymbol{1}\otimes\boldsymbol{a}_0)\boldsymbol{\psi}''
 + A_1\boldsymbol{\psi} = \boldsymbol{g} + \delta^2\tilde{\boldsymbol{f}}, \\
 \zeta + \psi_0' = g_{N+1} + \delta^2 \tilde{f}_{N+1}
\end{cases}
\end{equation}
satisfying $\zeta,\psi_0'\in B_e^\infty$, $\boldsymbol{\psi}\in B_0^\infty$, and $\psi_0(0)=0$, 
where $\tilde{f}_0,\tilde{\boldsymbol{f}},\tilde{f}_{N+1}$ are given by $f_0,\boldsymbol{f},f_{N+1}$ with 
$(\zeta, \psi_0, \boldsymbol{\psi})$ replaced by $(\tilde{\zeta},\tilde{\psi}_0,\tilde{\boldsymbol{\psi}})$. 
Moreover, the solution satisfies 
\[
\begin{cases}
 \opnorm{(\zeta,\psi_0,\boldsymbol{\psi})}_{2,\delta}
  \leq C + \delta C(M_2), \\
 \opnorm{(\zeta,\psi_0,\boldsymbol{\psi})}_{k,\delta}
  \leq C(k,M_{k-1}) + \delta C(M_2)M_k
\end{cases}
\]
for $k=3,4,5,\ldots$. 
In view of these estimates, we put $M_2 = 2C$ and define $M_k$ inductively by 
$M_k = 2C(k,M_{k-1})$ for $k=3,4,5,\ldots$. 
Then, by taking $\delta_0$ so small that $2\delta_0C(M_2)\leq C$ and $2\delta_0C(M_2) \leq 1$ 
we see that $(\zeta, \psi_0, \boldsymbol{\psi}) \in \boldsymbol{X}$. 
Therefore, if we define a map 
$\boldsymbol{\mathcal{T}}:(\tilde{\zeta},\tilde{\psi}_0,\tilde{\boldsymbol{\psi}})
 \mapsto (\zeta, \psi_0, \boldsymbol{\psi})$, 
then $\boldsymbol{\mathcal{T}}$ maps $\boldsymbol{X}$ into itself. 
Moreover, in exactly the same way as above, we obtain 
\[
\begin{cases}
 \opnorm{\boldsymbol{\mathcal{T}}(\tilde{\zeta},\tilde{\psi}_0,\tilde{\boldsymbol{\psi}})
  - \boldsymbol{\mathcal{T}}(\zeta,\psi_0,\boldsymbol{\psi})}_{2,\delta}
 \leq \frac12 \opnorm{(\tilde{\zeta},\tilde{\psi}_0,\tilde{\boldsymbol{\psi}})
   - (\zeta,\psi_0,\boldsymbol{\psi})}_{2,\delta}, \\
 \opnorm{\boldsymbol{\mathcal{T}}(\tilde{\zeta},\tilde{\psi}_0,\tilde{\boldsymbol{\psi}})
  - \boldsymbol{\mathcal{T}}(\zeta,\psi_0,\boldsymbol{\psi})}_{k,\delta} \\
 \quad\leq \frac12 \opnorm{(\tilde{\zeta},\tilde{\psi}_0,\tilde{\boldsymbol{\psi}})
   - (\zeta,\psi_0,\boldsymbol{\psi})}_{k,\delta} 
  + C(k,M_k) \opnorm{(\tilde{\zeta},\tilde{\psi}_0,\tilde{\boldsymbol{\psi}})
   - (\zeta,\psi_0,\boldsymbol{\psi})}_{k-1,\delta}
\end{cases}
\]
for $k=3,4,5,\ldots$. 
Therefore, by the contraction mapping principle, the map $\boldsymbol{\mathcal{T}}$ has a unique 
fixed point $(\zeta,\psi_0,\boldsymbol{\psi}) \in \boldsymbol{X}$, 
which is a solution to \eqref{red:IK4} and satisfies 
\[
\|\psi_0'\|_{k+2} + \|\zeta\|_{k+2,\delta} + \|\boldsymbol{\psi}\|_{k+3,\delta} \leq M_k
 \qquad\mbox{for}\quad k=2,3,4,\ldots
\]
with a constant $M_k$ independent of $\delta\in(0,\delta_0]$. 
The proof of Theorem \ref{intro:theorem} is complete.

\section{Numerical analysis for large amplitude solitary waves}
\label{section:na}
\setcounter{figure}{0}

In the previous section, we proved the existence of small amplitude solitary wave solutions to the 
Isobe--Kakinuma model \eqref{intro:IK}. 
In the present section, we will analyze numerically large amplitude solitary wave solutions to the model 
in the special case where the parameters are chosen as $N=1$ and $p_1=2$. 
Even in this special case, the Isobe--Kakinuma model gives a better approximation than the Green--Naghdi 
equations in the shallow water and strongly nonlinear regime. 
Therefore, we will consider the equations 
\begin{equation}\label{na:eq1}
\begin{cases}
 \displaystyle
  c\eta + H\phi_0' + \frac13 H^3 \phi_1' = 0, \\
 \displaystyle
  \phi_0'' + \frac15 H^2 \phi_1'' + 2\delta^{-2}\phi_1 = 0, \\
 \displaystyle
  c(\phi_0' + H^2 \phi_1') + \eta + \frac{1}{2} (\phi_0' + H^2 \phi_1')^2
   + 2\delta^{-2} H^2 \phi_1^2 = 0
\end{cases}
\end{equation}
under the boundary conditions at the spatial infinity 
\begin{equation}\label{na:BCI}
\eta(x), \phi_0'(x), \phi_1'(x), \phi_1(x) \to 0 \quad\mbox{as}\quad x\to\pm\infty
\end{equation}
and the symmetry 
\begin{equation}\label{na:sym}
\eta(-x)=\eta(x), \quad \phi_0(-x)=-\phi_0(x), \quad \phi_1(-x)=-\phi_1(x),
\end{equation}
where $H=1+\eta$. 
In this case, the constant $\gamma$ in Theorem \ref{intro:theorem} is given by $\gamma=\frac13$ 
so that we can put 
\begin{equation}\label{na:c}
c = 1+\frac23\delta^2
\end{equation}
and regard $\delta$ as a bifurcation parameter. 
Moreover, it follows from Proposition \ref{cl:prop2} that solutions to \eqref{na:eq1}--\eqref{na:BCI} satisfy 
\begin{equation}\label{na:cl}
\eta^2 - H(\phi_0')^2 - \frac23H^3\phi_0'\phi_1' - \frac15H^5(\phi_1')^2
 + \frac43\delta^{-2}H^3\phi_1^2 = 0.
\end{equation}

For numerical analysis, it is convenient to rewrite the equations in \eqref{na:eq1} as a system of 
ordinary differential equations of order 1. 
To this end, we introduce a new unknown function 
\begin{equation}\label{na:defu}
u = \phi_0' + H^2\phi_1',
\end{equation}
which is the horizontal component of the velocity on the water surface. 
Observe that the first equation in \eqref{na:eq1} and \eqref{na:defu} can be rewritten into a system 
\[
\begin{pmatrix}
 H & \frac13H^3 \\
 1 & H^2
\end{pmatrix}
\begin{pmatrix}
 \phi_0' \\
 \phi_1'
\end{pmatrix}
=
\begin{pmatrix}
 -c\eta \\
 u
\end{pmatrix},
\]
or equivalently
\begin{equation}\label{na:phi'}
\begin{pmatrix}
 \phi_0' \\
 \phi_1'
\end{pmatrix}
= \frac{1}{\frac23H^3}
\begin{pmatrix}
 -H^2(c\eta+\frac13Hu) \\
 c\eta+Hu
\end{pmatrix}.
\end{equation}
Differentiating the first equation in \eqref{na:eq1} and \eqref{na:defu}, we obtain 
\[
\begin{pmatrix}
 H & \frac13H^3 \\
 1 & H^2
\end{pmatrix}
\begin{pmatrix}
 \phi_0'' \\
 \phi_1''
\end{pmatrix}
=
\begin{pmatrix}
 -(c+u)\eta' \\
 u'-2H\phi_1'\eta'
\end{pmatrix},
\]
or equivalently
\[
\begin{pmatrix}
 \phi_0'' \\
 \phi_1''
\end{pmatrix}
= \frac{1}{\frac23H^3}
\begin{pmatrix}
 \bigl(-(c+u)H^2 + \frac23H^4\phi_1'\bigr)\eta' - \frac13H^3u' \\
 \bigl((c+u)-2H^2\phi_1'\bigr)\eta' + Hu'
\end{pmatrix}.
\]
Plugging these into the second equation in \eqref{na:eq1} yields 
\begin{equation}\label{na:eq2}
\left(-6(c+u)H + 2H^3\phi_1'\right)\eta' - H^2u' + 10\delta^{-2}H^2\phi_1 = 0.
\end{equation}
On the other hand, we can rewrite the third equation in \eqref{na:eq1} as
\[
cu + \eta + \frac12u^2 + 2\delta^{-2}H^2\phi_1^2 = 0.
\]
Differentiating this yields 
\begin{equation}\label{na:eq3}
(c+u)u' + (1+4\delta^{-2}H\phi_1^2)\eta' + 4\delta^{-2}H^2\phi_1 \phi_1' = 0.
\end{equation}
Now, it follows from \eqref{na:phi'}--\eqref{na:eq3} that 
\begin{equation}\label{na:eq4}
\begin{cases}
 \displaystyle
  \eta' = \frac{6H(c\eta+Hu)+10H^2(c+u)}{
   \delta^2\{ 6H(c+u)^2-3(c+u)(c\eta+Hu)-H^2(1+4H\phi_1^2) \}}\phi_1, \\[2ex]
 \displaystyle
  u' = -\frac{18(c\eta+Hu)\big(2H(c+u)-(c\eta+Hu)\big)+10H^3(1+4H\phi_1^2)}{
   \delta^2H\{ 6H(c+u)^2-3(c+u)(c\eta+Hu)-H^2(1+4H\phi_1^2) \}}\phi_1, \\[2ex]
 \displaystyle
 \phi_1' = \frac{3}{2H^3}(c\eta+Hu).
\end{cases}
\end{equation}
To summarize, \eqref{na:eq1}--\eqref{na:sym} has been transformed equivalently into \eqref{na:eq4} 
under the boundary conditions at the spatial infinity 
\begin{equation}\label{na:BCI2}
\eta(x), u(x), \phi_1(x) \to 0 \quad\mbox{as}\quad x\to\pm\infty
\end{equation}
and the symmetry 
\begin{equation}\label{na:sym2}
\eta(-x)=\eta(x), \quad u(-x)=u(x), \quad \phi_1(-x)=-\phi_1(x).
\end{equation}

\begin{proposition}\label{na:prop1}
Any regular solution $(\eta,u,\phi_1)$ to \eqref{na:eq4}--\eqref{na:BCI2} satisfies the two identities 
\begin{equation}\label{na:id}
\begin{cases}
 \displaystyle
  cu + \eta + \frac{1}{2}u^2 + 2\delta^{-2}H^2\phi_1^2 = 0, \\
 \displaystyle
  \eta^2 - Hu^2 + 2u(c\eta+Hu) - \frac{6}{5H}(c\eta+Hu)^2 + \frac43\delta^{-2}H^3\phi_1^2 = 0.
\end{cases}
\end{equation}
\end{proposition}

\begin{proof}[{\bf Proof.}]
The first identity is nothing but the third equation in \eqref{na:eq1}. 
Plugging \eqref{na:phi'} into \eqref{na:cl}, we obtain the second one. 
\end{proof}

In order to obtain numerical solutions to \eqref{na:eq4}--\eqref{na:sym2}, it is sufficient to determine 
the initial data $(\eta(0),u(0),\phi_1(0))$. 
It follows from \eqref{na:sym2} that $\phi_1(0)=0$, which together with the identities in \eqref{na:id} 
implies 
\begin{equation}\label{na:id2}
\begin{cases}
 \displaystyle
  cu(0) + \eta(0) + \frac{1}{2}u(0)^2 = 0, \\
 \displaystyle
  \eta(0)^2 - H(0)u(0)^2 + 2u(0)\bigl( c\eta(0) + H(0)u(0) \bigr)
   - \frac{6}{5H(0)}\bigl( c\eta(0) + H(0)u(0) \bigr)^2 = 0,
\end{cases}
\end{equation}
where $H(0)=1+\eta(0)$ and $c$ is given by \eqref{na:c}. 
Particularly, by eliminating $\eta(0)$ from these two identities we obtain 
\begin{equation}\label{na:id3}
7u(0)^4+42cu(0)^3+6(13c^2-3)u(0)^2+8c(13c^2-8)u(0)+8(6c^2-1)(c^2-1) = 0.
\end{equation}
This is a quartic equation in $u(0)$ so that for each given $c$ we have four roots. 
Generally, two of them are complex numbers and one real root does not give the correct initial data 
for the solitary wave solution, so that we can determine the initial data $(\eta(0),u(0),\phi_1(0))$ 
for appropriately chosen $\delta$.

We proceed to compare solitary wave solutions to the Isobe--Kakinuma model \eqref{na:eq1}--\eqref{na:c} 
calculated numerically as above with the classical solitons of the Korteweg--de Vries equation 
\begin{equation}\label{na:kdv}
\eta_{\rm KdV}(x) = \frac43\delta^2\mbox{\rm sech}^2x,
\end{equation}
which is the first approximation for small amplitude solitary wave solutions to the Isobe--Kakinuma model 
as was guaranteed by Theorem \ref{intro:theorem}. 
In Figure \ref{na:IKvsKdV}, we plot the surface profile $\eta(x)$ of the solitary wave solutions 
to the Isobe--Kakinuma model \eqref{na:eq1} and the soliton $\eta_{\rm KdV}(x)$ of the 
Korteweg--de Vries equation given by \eqref{na:kdv} for several values of $\delta$. 
We observe that for small $\delta$ the error $\|\eta-\eta_{\rm KdV}\|_{L^\infty}$ is small as expected. 
As $\delta$ increases, the error cannot be negligible anymore and the wave height of the solitary wave 
to the Isobe--Kakinuma model becomes larger than that of the classical soliton. 
We can catch numerically the solitary wave solutions to the Isobe--Kakinuma model for $\delta$ up to 
some critical value $\delta_c$. 
For $\delta$ beyond this critical value $\delta_c$, the quartic equation \eqref{na:id3} does not have 
any real root so that the solitary wave solution might not exist. 
This suggests that there exists a maximum height of the solitary wave solutions to the Isobe--Kakinuma model 
as in the case of the full water wave problem.

\begin{figure}[h]
\setlength{\unitlength}{1pt}
\begin{picture}(0,0)
\put(11,-20){(a)}
\put(160,-20){(b)}
\put(308,-20){(c)}
\put(11,-127){(d)}
\put(160,-127){(e)}
\put(308,-127){(f)}
\end{picture}
\begin{center}
\begin{tabular}{ccc}
\includegraphics[width=0.3\linewidth]{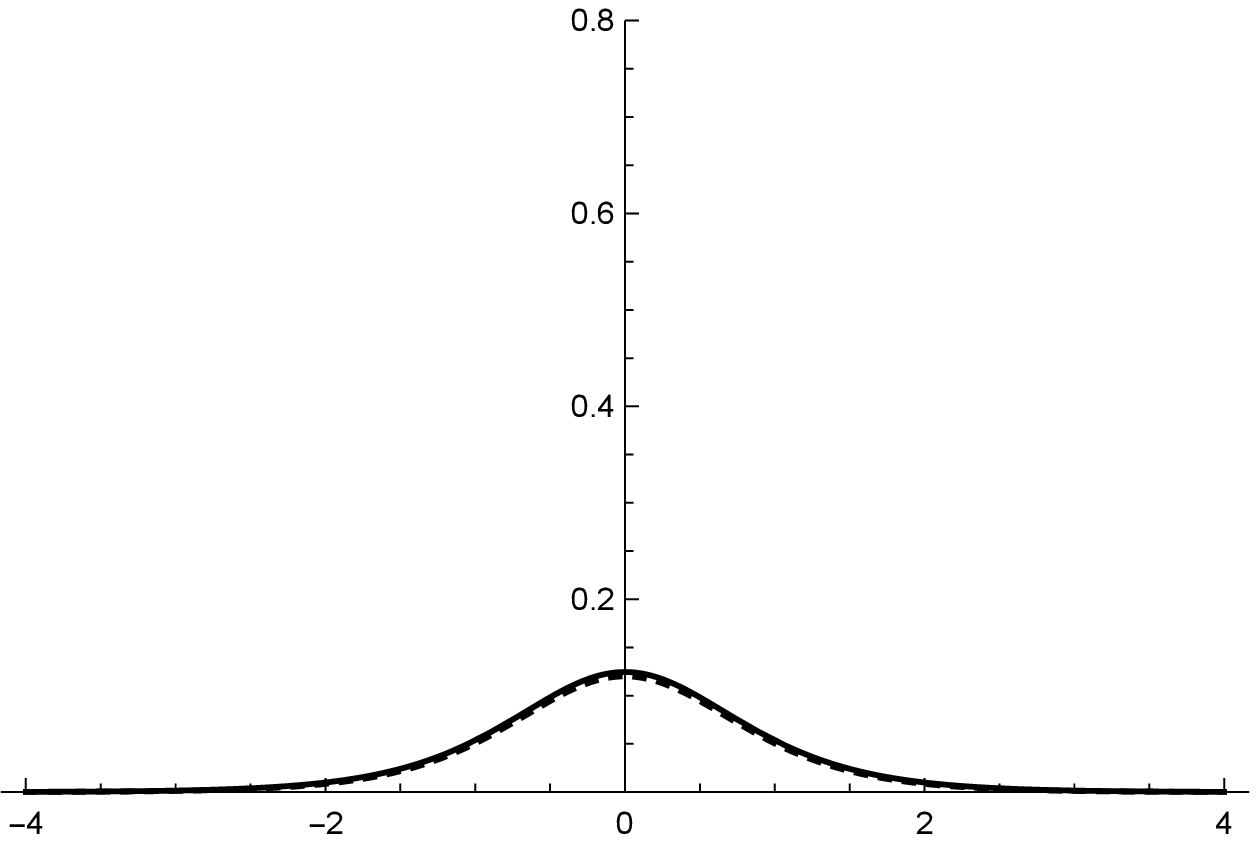} &
\includegraphics[width=0.3\linewidth]{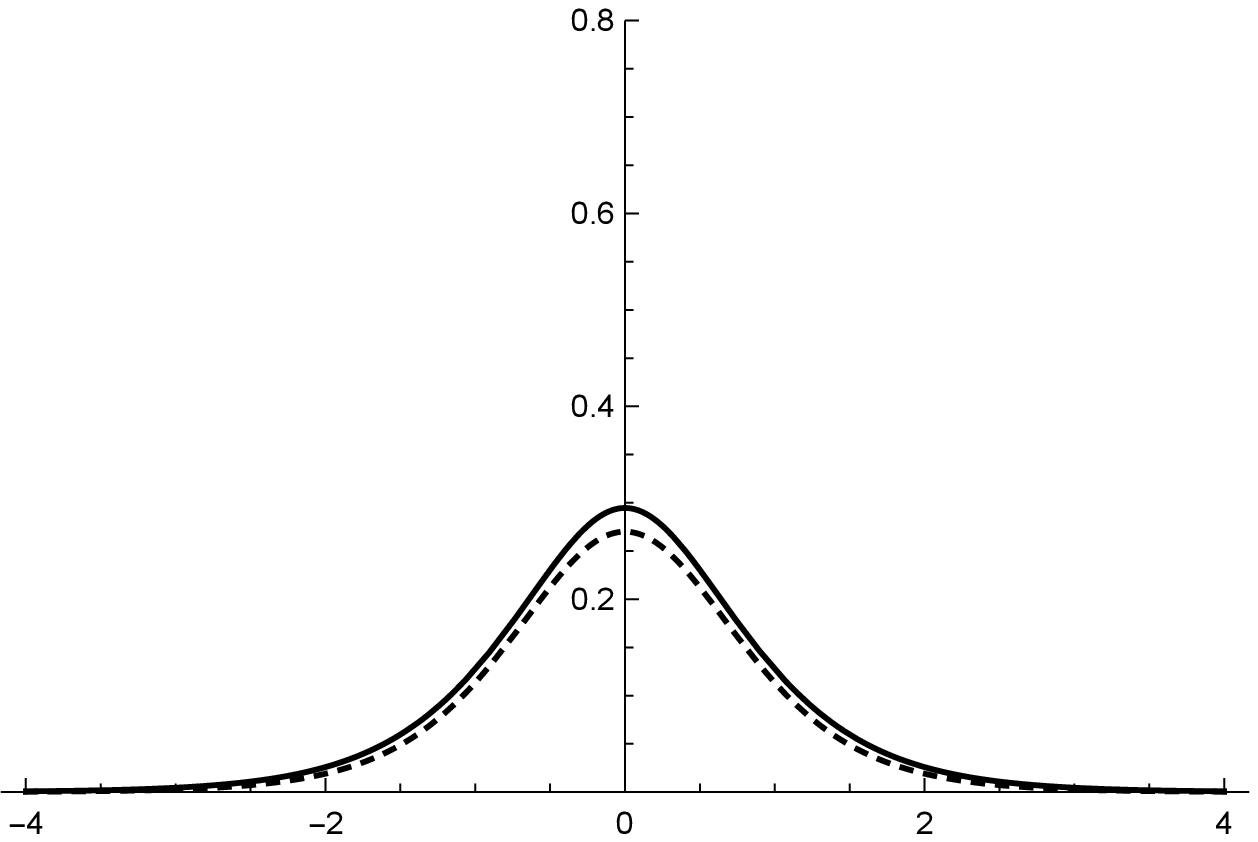} &
\includegraphics[width=0.3\linewidth]{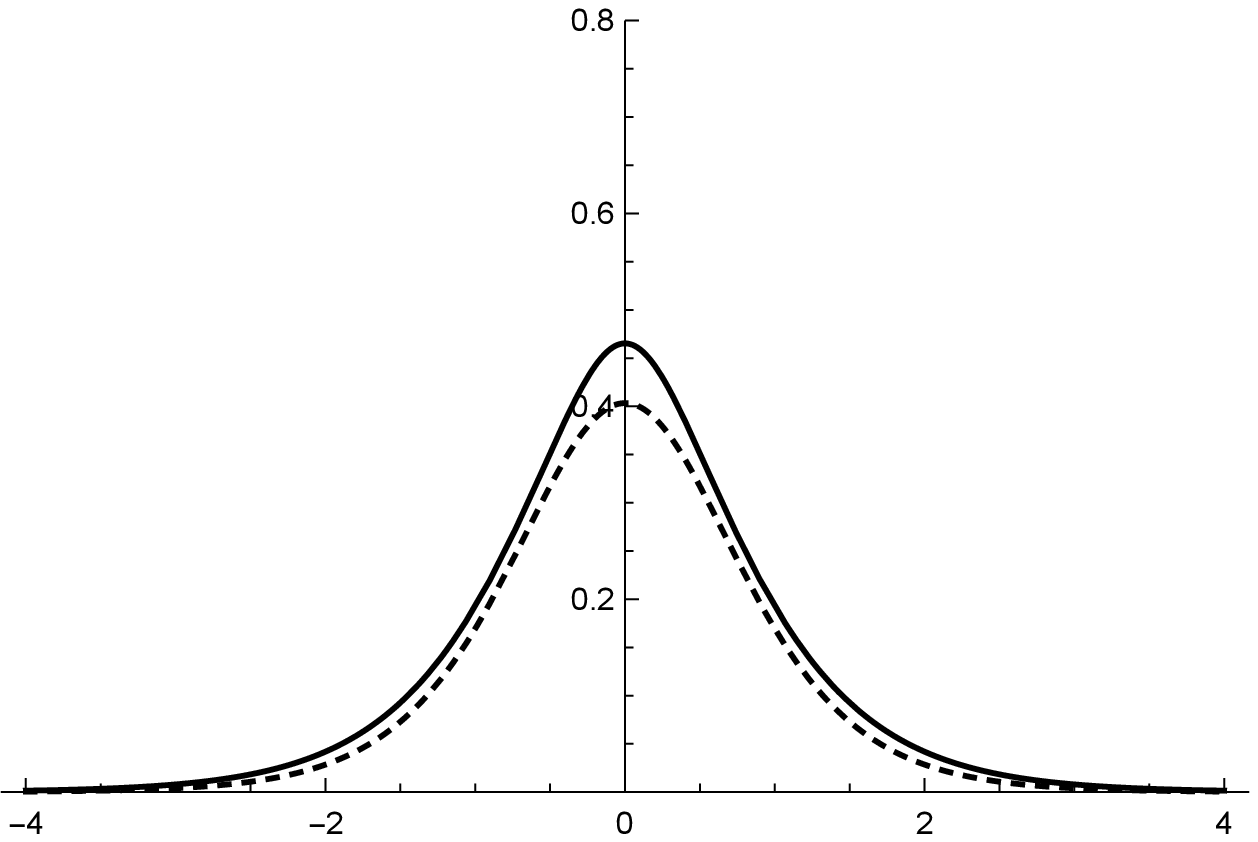} 
\end{tabular}

\bigskip
\begin{tabular}{ccc}
\includegraphics[width=0.3\linewidth]{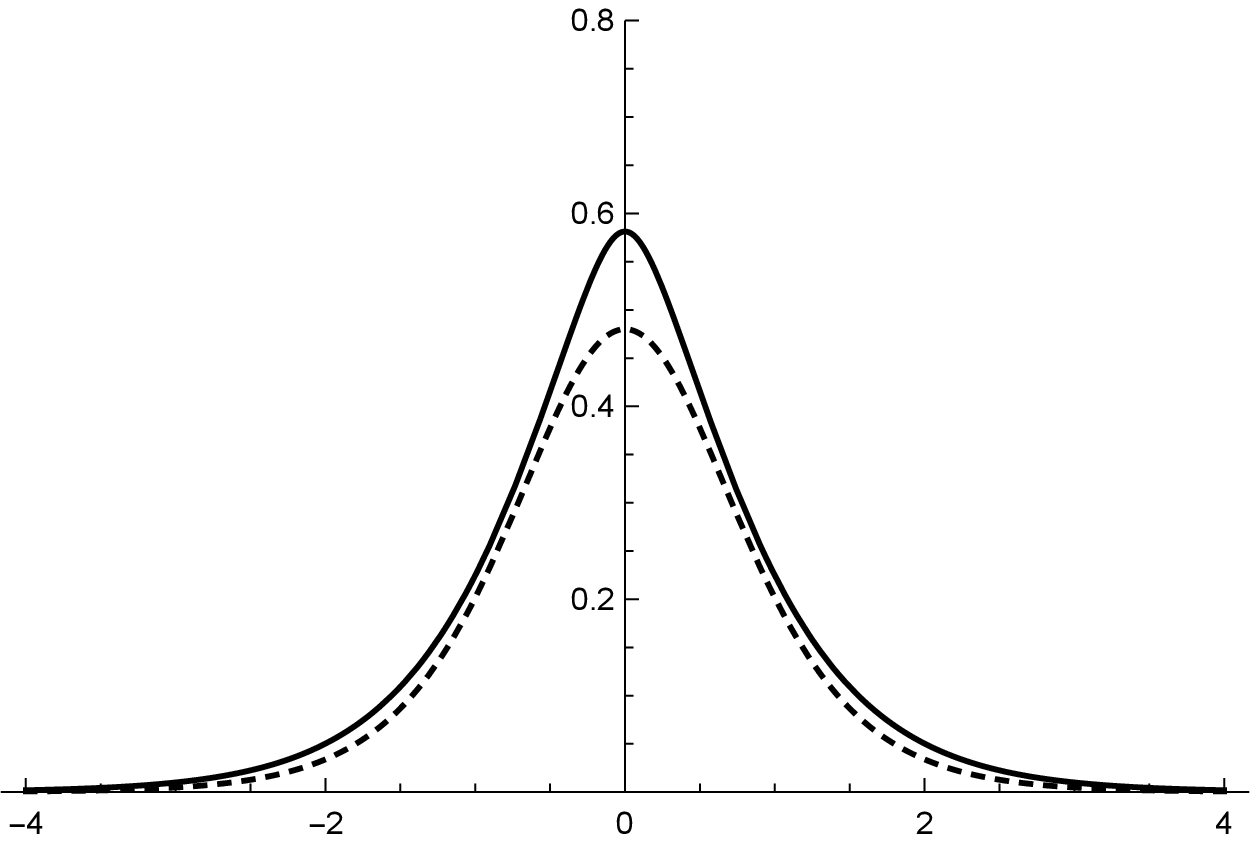} &
\includegraphics[width=0.3\linewidth]{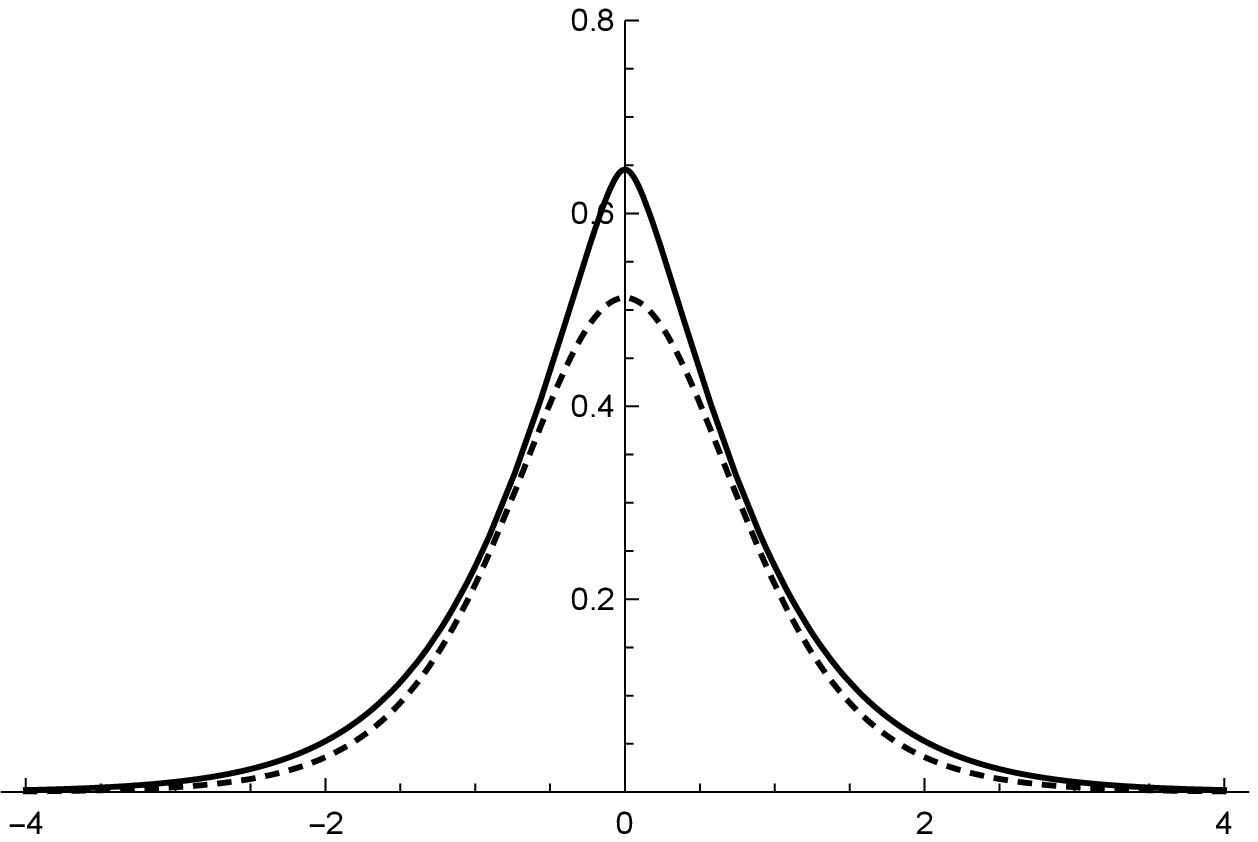} &
\includegraphics[width=0.3\linewidth]{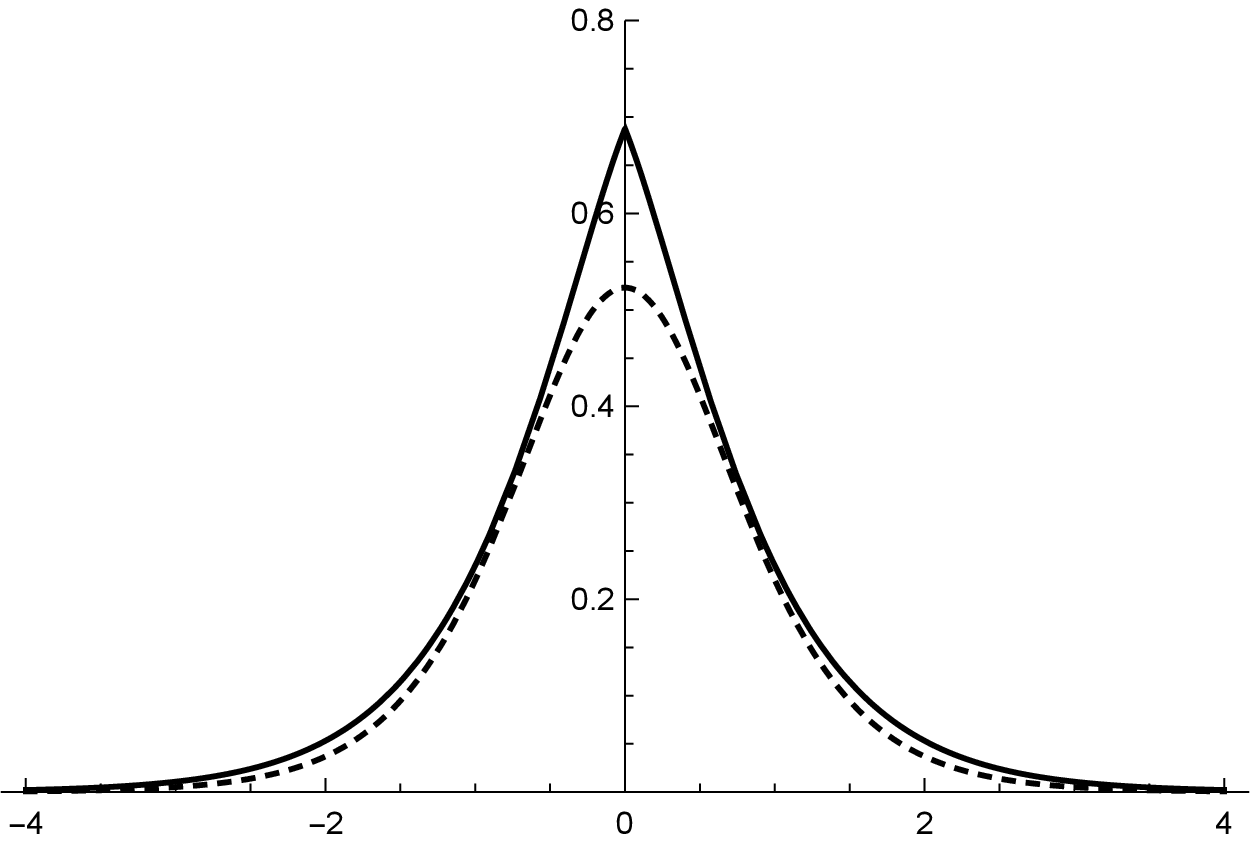} 
\end{tabular}
\end{center}
\caption{Surface profiles of the solitary wave solutions to the Isobe--Kakinuma model (solid line) and 
the Korteweg--de Vries equation (dashed line) for several values of $\delta$: (a) $\delta=0.3$, 
(b) $\delta=0.45$, (c) $\delta=0.55$, (d) $\delta=0.6$, (e) $\delta=0.62$, (f) $\delta=0.62633493$. }
\label{na:IKvsKdV}
\end{figure}

In Figures \ref{na:crests1} and \ref{na:crests2}, we plot also the surface profiles of the solitary wave 
solutions to the Isobe--Kakinuma model for several values of $\delta$ in the same figures. 
We observe that the wave height is monotonically increasing as $\delta$ increases and that a sharp crest 
is formed as $\delta$ approaches the critical value $\delta_c$. 
We will analyze more precisely this formation of a sharp crest. 
To this end, we calculate numerically the curvature $\kappa$ at the crest of the surface profile, 
which is defined by 

\begin{figure}[t]
 \begin{minipage}{0.48\hsize}
  \begin{center}
   \includegraphics[width=70mm]{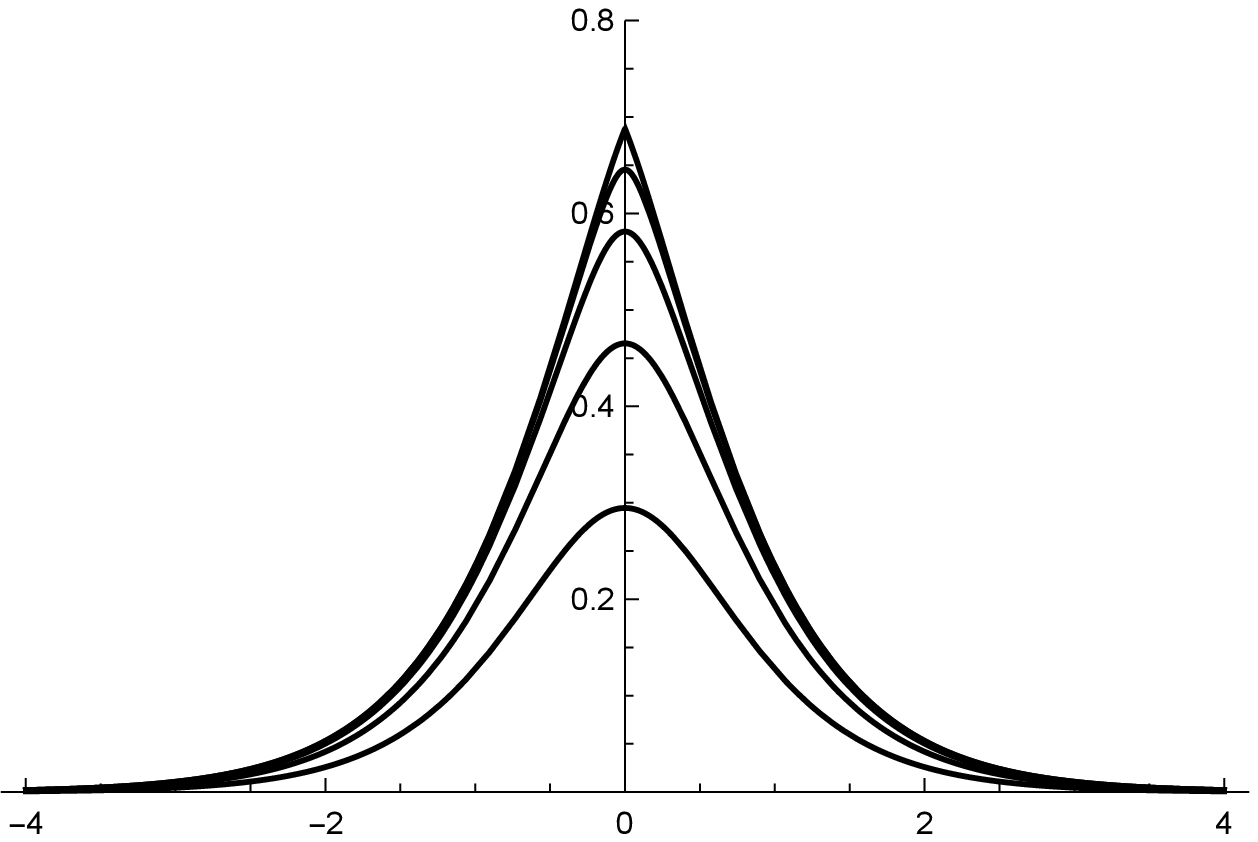}
  \end{center}
  \caption{Surface profiles of the solitary wave solutions to the Isobe--Kakinuma model for 
  $\delta=0.45$, 0.55, 0.6, 0.62, 0.62633493.}
  \label{na:crests1}
 \end{minipage}\quad
 \begin{minipage}{0.49\hsize}
  \begin{center}
   \includegraphics[width=70mm]{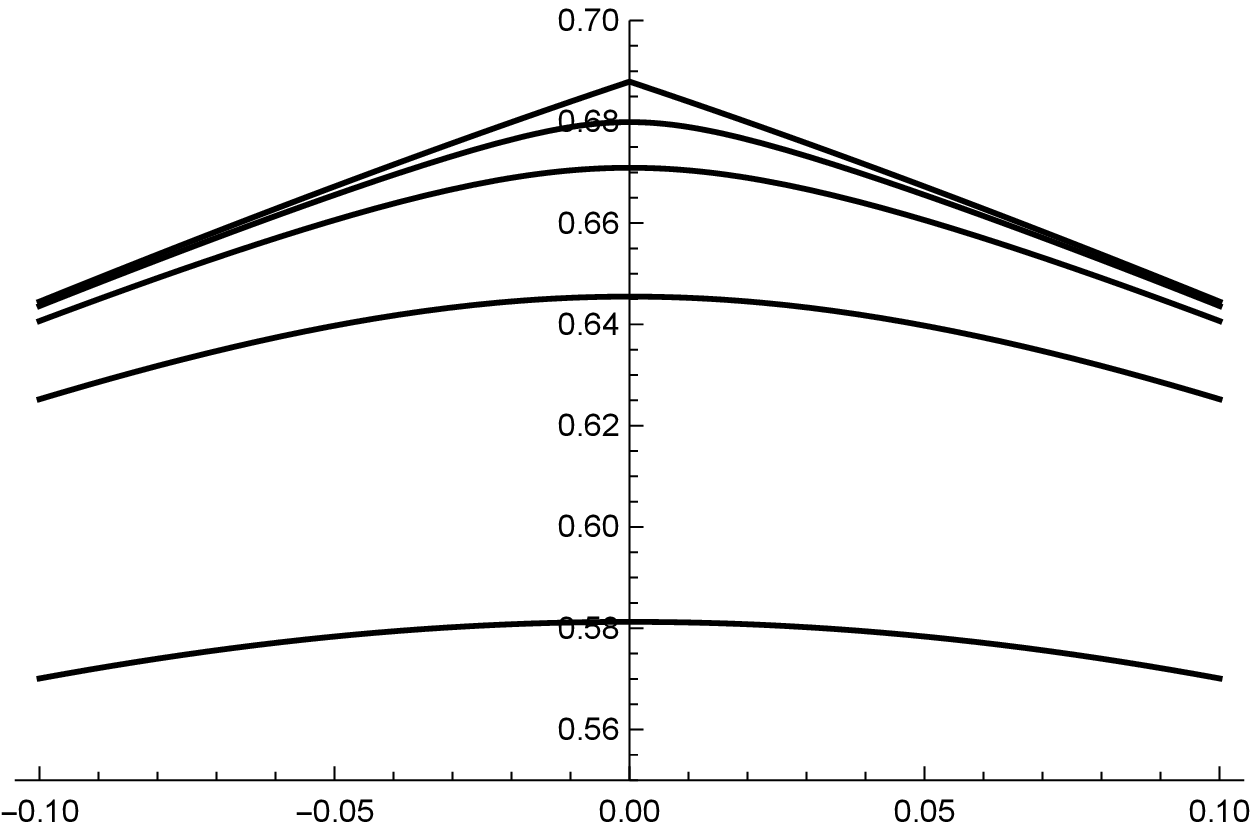}
  \end{center}
  \caption{Surface profiles near the crests of the solitary wave solutions to the Isobe--Kakinuma model for 
  $\delta=0.6$, 0.62, 0.625, 0.626, 0.62633493.}
  \label{na:crests2}
 \end{minipage}
\end{figure}

\[
\kappa(x) = \frac{\eta''(x)}{\bigl( 1+(\eta'(x))^2 \bigr)^\frac32}.
\]
We also introduce a function $d(x)$ by 
\begin{equation}\label{na:deno}
d = 6H(c+u)^2-3(c+u)(c\eta+Hu)-H^2(1+4H\phi_1^2),
\end{equation}
where $H=1+\eta$ and $c$ is given by \eqref{na:c}. 
This function appears in the denominator of the right-hand side of the ordinary differential equations 
for $\eta$ and $u$ in \eqref{na:eq4}. 
In Table \ref{na:list}, we list the wave height, the curvature at the crest, and the value of the 
denominator $d$ at the crest of the solitary wave solutions to the Isobe--Kakinuma model 
for several values of $\delta$. 
We observe that as $\delta$ approaches some critical value $\delta_c$, the wave height converges toward a maximum 
wave height $\eta_c(0)$, the curvature at the crest is blowing up, and the denominator at the crest is 
going to vanish. 
This consideration suggests strongly the existence of solitary wave of extreme form as well as a sharp crest 
to the Isobe--Kakinuma model and that the critical value $\delta_c$ would be obtained by the equation 
$d(0)=0$, that is, 
\begin{equation}\label{na:id4}
6H(0)\bigl( c+u(0) \bigr)^2 - 3\bigl( c+u(0) \bigr)\bigl( c\eta(0)+H(0)u(0) \bigr) - H(0)^2 = 0,
\end{equation}
where $H(0)=1+\eta(0)$ and $c$ is given by \eqref{na:c}. 
In fact, we can calculate this critical value $\delta_c$, the maximum wave height $\eta_c(0)$, 
the horizontal velocity $u_c(0)$ of the water at the crest, and the critical phase speed $c_c$ by 
solving nonlinear algebraic equations \eqref{na:id3} and \eqref{na:id4} together with \eqref{na:c}. 
Those values are approximately given by 
\begin{equation}\label{na:cv}
\delta_c = 0.62633493, \quad \eta_c(0) = 0.687926, \quad u_c(0) = -0.797196, \quad c_c = 1.26153.
\end{equation}
In Figures \ref{na:extreme1} and \ref{na:extreme2}, we plot the surface and horizontal velocity profiles 
of the solitary wave of extreme form to the Isobe--Kakinuma model. 
It follows from \eqref{na:cv} that 
$c_c + u_c(0)>0$, which means that the crest of the solitary wave of extreme form is not the stagnation point 
unlike the full water wave problem. 

\begin{table}[t]
\begin{center}
 \begin{tabular}{lllllll}
  \hline
  $\delta$   & & $\eta(0)$ & & $-\kappa(0)$                   & & $d(0)$ \\ \hline\hline
  0.6        & & 0.581258  & & 2.34087                        & & 1.55722 \\
  0.62       & & 0.645485  & & 4.85676                        & & 7.30167$\times 10^{-1}$ \\
  0.625      & & 0.670918  & & 1.04536$\times 10$             & & 3.23799$\times 10^{-1}$ \\
  0.626      & & 0.679938  & & 2.0651\phantom{0}$\times 10$   & & 1.59473$\times 10^{-1}$ \\
  0.6263     & & 0.685463  & & 6.3354\phantom{0}$\times 10$   & & 5.08746$\times 10^{-2}$ \\
  0.62633    & & 0.687014  & & 1.68098$\times 10^2$           & & 1.90423$\times 10^{-2}$ \\
  0.626334   & & 0.687532  & & 3.8648\phantom{0}$\times 10^2$ & & 8.26255$\times 10^{-3}$ \\
  0.6263349  & & 0.687855  & & 2.12563$\times 10^3$           & & 1.5\phantom{0000}$\times 10^{-3}$ \\
  0.62633493 & & 0.687915  & & 1.384\phantom{00}$\times 10^4$ & & 2.30314$\times 10^{-4}$ \\ \hline
 \end{tabular}
\caption{List of the wave height, the curvature at the crest, and the value of the denominator $d$ at 
the crest of the solitary wave solutions to the Isobe--Kakinuma model.}
\label{na:list}
\end{center}
\end{table}

\begin{figure}[h]
 \begin{minipage}{0.48\hsize}
  \begin{center}
   \includegraphics[width=70mm]{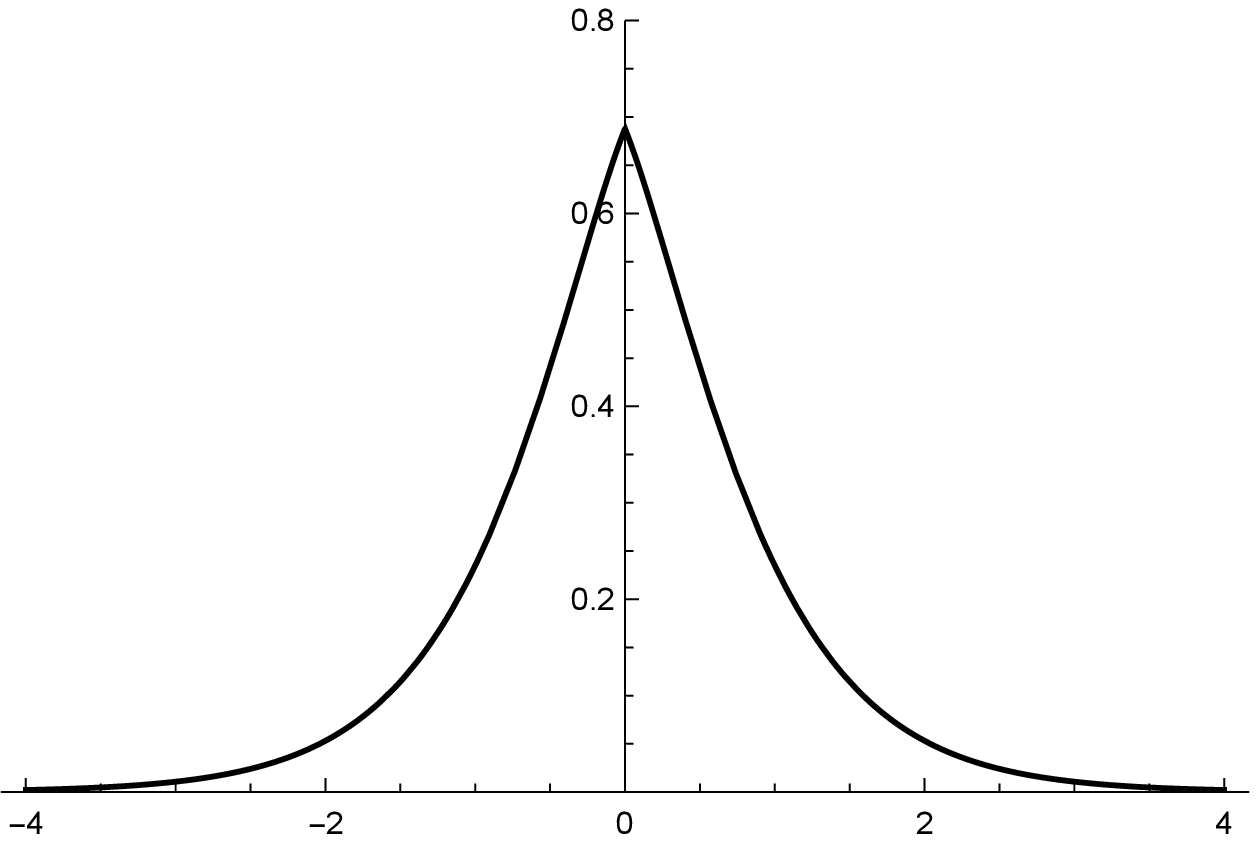}
  \end{center}
  \caption{Surface profile of the solitary wave of extreme form to the Isobe--Kakinuma model.}
  \label{na:extreme1}
 \end{minipage}\quad
 \begin{minipage}{0.49\hsize}
  \begin{center}
   \includegraphics[width=70mm]{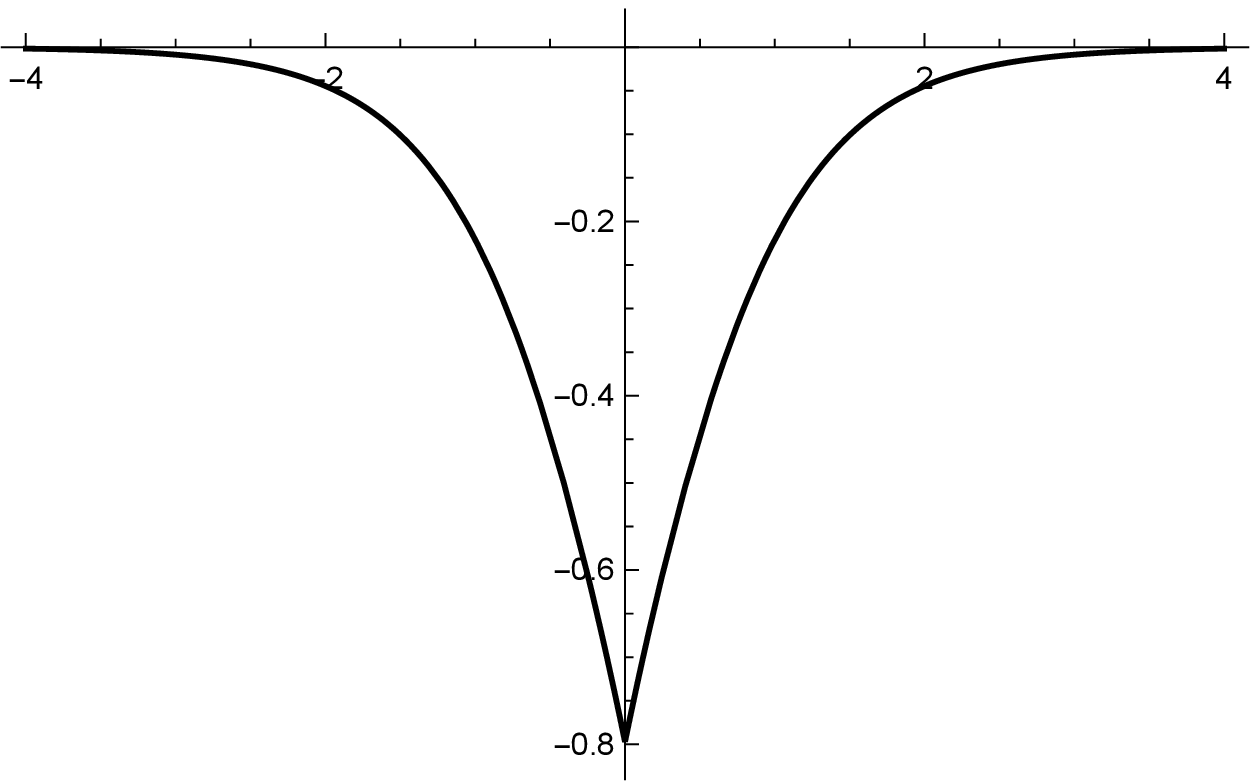}
  \end{center}
  \caption{Horizontal velocity profile of the solitary wave of extreme form to the Isobe--Kakinuma model.}
  \label{na:extreme2}
 \end{minipage}
\end{figure}

We proceed to calculate the angle of the crest of the solitary wave of extreme form. 
To this end, it is sufficient to evaluate $\eta_c'(\pm0)$, where $(\eta_c,u_c,\phi_{1c})$ is the 
solution to the Isobe--Kakinuma model \eqref{na:eq4} in the critical case $\delta=\delta_c$. 
We denote by $d_c$ the corresponding denominator defined by \eqref{na:deno} so that 
by the first equation in \eqref{na:eq4} we have 
\begin{equation}\label{na:ang1}
\eta_c' = \frac{ 2H_c(8H_cv_c-3c) }{\delta_c^2}\frac{\phi_{1c}}{d_c},
\end{equation}
where $H_c=1+\eta_c$ and 
\[
v_c=c_c+u_c.
\]
It follows from the third equation in \eqref{na:eq4} that 
\begin{equation}\label{na:ang2}
\phi_{1c}'(0) = \frac{3}{2H_c(0)^3}\bigl( H_c(0)v_c(0)-c_c \bigr).
\end{equation}
Differentiating \eqref{na:deno} and using $\phi_{1c}(0)=0$, we have 
\[
d_c'(\pm0) = \bigl( 3v_c(0)^2-2H_c(0) \bigr)\eta_c'(\pm0) + 3\bigl( 2H_c(0)v_c(0) + c_c \bigr)u_c'(\pm0).
\]
It follows from \eqref{na:eq3} in the critical case $\delta=\delta_c$ that 
\[
v_c(0)u_c'(\pm0) + \eta_c'(\pm0) = 0,
\]
so that 
\begin{equation}\label{na:ang3}
d_c'(\pm0) = \left( 3v_c(0)^2 - 8H_c(0) + \frac{3c_c}{v_c(0)} \right)\eta_c'(\pm0).
\end{equation}
By \eqref{na:ang2}, \eqref{na:ang3}, and l'H\^opital's rule, we obtain 
\begin{align*}
\lim_{x\to\pm0}\frac{\phi_{1c}(x)}{d_c(x)} 
&= \lim_{x\to\pm0}\frac{\phi_{1c}'(x)}{d_c'(x)} \\
&= \frac{3v_c(0)\bigl( H_c(0)v_c(0)-c_c \bigr)}{ 
 2H_c(0)^3\bigl( 3v_c(0)^3-8H_c(0)v_c(0)-3c_c \bigr)}\frac{1}{\eta_c'(\pm0)}.
\end{align*}
Therefore, passing to the limit $x\to\pm0$ in \eqref{na:ang1} yields 
\begin{equation}\label{na:ang4}
\eta_c'(\pm0) = \mp\sqrt{\frac{3v_c(0)\bigl( H_c(0)v_c(0)-c_c \bigr)\bigl( 8H_c(0)v_c(0)-3c_c\bigr)}{ 
 \delta_c^2H_c(0)^2\bigl( 3v_c(0)^3-8H_c(0)v_c(0)-3c_c \bigr)}},
\end{equation}
which gives the angle of the crest.

Here, note that we have rewritten all physical quantities in a nondimensional form. 
In order to calculate the angle of the crest in the physical space, 
we have to work with dimensional variables. 
Let $x^*$ and $\eta^*$ be the horizontal spatial coordinate and the surface elevation in the physical space, 
respectively, so that we have $x^*=\lambda x$ and $\eta^*=h\eta$, where $h$ is the mean depth of the water 
and $\lambda$ the typical wavelength. 
The angle of the crest in the physical space should be calculated from ${\eta_c^*}'(\pm0)$. 
In view of the relation $\eta^*(x^*)=h\eta(\lambda^{-1}x^*)$, we have 
${\eta^*}'(x^*)=\delta\eta'(x)$ so that ${\eta_c^*}'(\pm0)=\delta_c\eta_c'(\pm0)$, that is, 
\begin{align*}
{\eta_c^*}'(\pm0) 
&= \mp\sqrt{\frac{3v_c(0)\bigl( H_c(0)v_c(0)-c_c \bigr)\bigl( 8H_c(0)v_c(0)-3c_c \bigr)}{
 H_c(0)^2\bigl( 3v_c(0)^3-8H_c(0)v_c(0)-3c_c \bigr)}} \\
&= 0.24397.
\end{align*}
Now, the included angle $\theta$ of the shape crest in the physical space is approximately given by 
$\theta=\ang{152.6}$, which is larger than the included angle $\ang{120}$ of the sharp crest of 
the solitary wave solution of extreme form to the full water wave problem.


\bigskip
Mathieu Colin \par
{\sc University of Bordeaux, CNRS, Bordeaux INP, } \par
{\sc  IMB, UMR 5251, INRIA CARDAMOM, F-33400, Talence, France} \par
E-mail: Mathieu.Colin@math.u-bordeaux.fr

\bigskip
Tatsuo Iguchi \par
{\sc Department of Mathematics} \par
{\sc Faculty of Science and Technology, Keio University} \par
{\sc 3-14-1 Hiyoshi, Kohoku-ku, Yokohama, 223-8522, Japan} \par
E-mail: iguchi@math.keio.ac.jp


\begin{thebibliography}{99}
%
\bibitem{AmickFraenkelToland1982}
C. J. Amick, L. E. Fraenkel, and J. F. Toland, 
On the Stokes conjecture for the wave of extreme form, 
Acta Math., {\bf 148} (1982), 193--214.
%
\bibitem{AmickToland1981}
C. J. Amick and J. F. Toland, 
On solitary water-waves of finite amplitude, 
Arch. Rational Mech. Anal., {\bf 76} (1981), 9--95.
%
\bibitem{FriedrichsHyers1954}
K. O. Friedrichs and D. H. Hyers, 
The existence of solitary waves, 
Comm. Pure Appl. Math., {\bf 7} (1954), 517--550. 
%
\bibitem{Stokes1847}
G. G. Stokes, 
On the theory of oscillatory waves, 
Trans. Cambridge Philos. Soc., {\bf 8} (1847), 441--455. 
%
\bibitem{Iguchi2018-1}
T. Iguchi, 
Isobe--Kakinuma model for water waves as a higher order shallow water approximation, 
J. Differential Equations, {\bf 265} (2018), 935--962.
%
\bibitem{Iguchi2018-2}
T. Iguchi, 
A mathematical justification of the Isobe--Kakinuma model for water waves with and without bottom topography, 
J. Math. Fluid Mech., {\bf 20} (2018), 1985--2018. 
%
\bibitem{Isobe1994}
M. Isobe, 
A proposal on a nonlinear gentle slope wave equation, 
Proceedings of Coastal Engineering, 
Japan Society of Civil Engineers, {\bf 41} (1994), 1--5 [Japanese]. 
%
\bibitem{Isobe1994-2}
M. Isobe, 
Time-dependent mild-slope equations for random waves, 
Proceedings of 24th International Conference on Coastal Engineering, ASCE, 285--299, 1994.
%
\bibitem{Kakinuma2000}
T. Kakinuma, 
[title in Japanese], 
Proceedings of Coastal Engineering, 
Japan Society of Civil Engineers, {\bf 47} (2000), 1--5 [Japanese]. 
%
\bibitem{Kakinuma2001}
T. Kakinuma, 
A set of fully nonlinear equations for surface and internal gravity waves, 
Coastal Engineering V: Computer Modelling of Seas and Coastal Regions, 
225--234, WIT Press, 2001. 
%
\bibitem{Kakinuma2003}
T. Kakinuma, 
A nonlinear numerical model for surface and internal waves shoaling on a permeable beach, 
Coastal engineering VI: Computer Modelling and Experimental Measurements of Seas and Coastal Regions, 
227--236, WIT Press, 2003.
%
\bibitem{LannesMarche2016}
D. Lannes and F. Marche,
Nonlinear wave-current interactions in shallow water, 
Stud. Appl. Math., {\bf 136} (2016), 382--423. 
%
\bibitem{Luke1967}
J. C. Luke, 
A variational principle for a fluid with a free surface, 
J. Fluid Mech., {\bf 27} (1967), 395--397. 
%
\bibitem{MurakamiIguchi2015}
Y. Murakami and T. Iguchi, 
Solvability of the initial value problem to a model system for water waves, 
Kodai Math. J., {\bf 38} (2015), 470--491.
%
\bibitem{NemotoIguchi2017}
R. Nemoto and T. Iguchi, 
Solvability of the initial value problem to the Isobe--Kakinuma model for water waves, 
J. Math. Fluid Mech., {\bf 20} (2018), 631--653. 
%
\bibitem{Varvaruca2009}
E. Varvaruca, 
On the existence of extreme waves and the Stokes conjecture with vorticity, 
J. Differential Equations, {\bf 246} (2009) 4043--4076. 
\end{thebibliography}
\end{document}